\documentclass[a4paper,11pt]{article}
\usepackage{pdfpages}

\usepackage[margin=1in]{geometry}  
\usepackage{amsmath}               
\usepackage{amsfonts}              
\usepackage{amsthm}      
\usepackage{amssymb}
\usepackage{bbm}

\usepackage[utf8]{inputenc}
\usepackage{mathtools}
\usepackage{cite}
\def\acts{\curvearrowright}

\usepackage{color}   
\usepackage{hyperref}
\hypersetup{
    colorlinks=true, 
    linktoc=all,     
    linkcolor=black,  
}

\DeclarePairedDelimiter\abs{\lvert}{\rvert}%
\DeclarePairedDelimiter\norm{\lVert}{\rVert}%

\makeatletter
\let\oldabs\abs
\def\abs{\@ifstar{\oldabs}{\oldabs*}}
\let\oldnorm\norm
\def\norm{\@ifstar{\oldnorm}{\oldnorm*}}

\makeatother

\theoremstyle{definition}
\newtheorem{thm}{Theorem}
\newtheorem{lem}{Lemma}
\newtheorem{prop}{Proposition}

\newtheorem{defn}{Definition}
\newtheorem{example}{Example}

\newtheorem{remark}{Remark}

\newtheorem{question}{Question}

\DeclareMathOperator{\id}{id}

\DeclareMathOperator{\Aut}{Aut}
\DeclareMathOperator{\Isom}{Isom}
\DeclareMathOperator{\stab}{stab}
\DeclareMathOperator{\support}{support}

\DeclareMathOperator{\Sym}{Sym}
\DeclareMathOperator{\Cay}{Cay}

\let\phi\varphi

\let\empt\varnothing

\newcommand{\EE}{\mathbb{E}}      

\newcommand{\RR}{\mathbb{R}}      

\newcommand{\e}{\varepsilon} 
\newcommand{\ZZ}{\mathbb{Z}}      
\newcommand{\PP}{\mathbb{P}}      
\newcommand{\MM}{\mathbb{M}}      
\newcommand{\Maper}{\mathbb{M}^\text{aper}}

\newcommand{\NN}{\mathbb{N}}      
\newcommand{\HH}{\mathbb{H}}      

\newcommand{\1}{\mathbbm{1}}

\newcommand\restr[2]{{
  \left.\kern-\nulldelimiterspace 
  #1 
  \vphantom{\big|} 
  \right|_{#2} 
  }}

\usepackage{color}
\usepackage{mathrsfs}
\newcommand{\Rel}{\mathcal{R}} 
\newcommand{\MMo}{{\mathbb{M}_0}}
\newcommand{\Marrow}{\overrightarrow{\mathbb{M}_0}} 
\newcommand{\muarrow}{\overrightarrow{\mu_0}}

\newcommand{\PPP}{\mathscr{P}}

\DeclareMathOperator{\graph}{Graph}

\DeclareMathOperator{\intensity}{intensity}

\begin{document}

\title{The Palm groupoid of a point process and factor graphs on amenable and Property (T) groups}

\author{Sam Mellick}

\maketitle
 
 \begin{abstract}
     
     We define a probability measure preserving and $r$-discrete groupoid that is associated to every invariant point process on a locally compact and second countable group. This groupoid governs certain factor processes of the point process, in particular the existence of Cayley factor graphs. With this method we are able to show that point processes on amenable groups admit all (and only admit) Cayley factor graphs of amenable groups, and that the Poisson point process on groups with Kazhdan's Property (T) admits no Cayley factor graphs. This gives examples of pmp countable Borel equivalence relations that cannot be generated by any free action of a countable group.
     
 \end{abstract}
 

\section*{Introduction}

This paper discusses \emph{invariant point processes} on locally compact and second countable groups. The reader is not assumed to have any familiarity with point process theory, only the most basic probability is required. If $G$ is such a group, then a point process is a random\footnote{We will be interested in point processes as examples of \emph{actions} of groups, but we will nevertheless use probability theoretic terminology where possible as it is a more elegant language for expressing many things. See Definition \ref{basicdefs} for a review of the terminology.} closed and discrete subset $\Pi \subset G$, and it is called \emph{invariant} if its distribution is unchanged by translation -- that is, the distribution of $\Pi$ is the same as that of $g\Pi$ for all $g \in G$. If $G$ is a discrete group, then this is nothing other than an invariant percolation on the group. We concern ourselves with the case of \emph{nondiscrete} groups. 

A \emph{factor graph} on such a point process $\Pi$ is a deterministically and measurably constructed graph $\mathscr{G}(\Pi)$ with vertex set $\Pi$. This graph should be equivariant in the sense that $\mathscr{G}(g\Pi) = g \mathscr{G}(\Pi)$ -- that is, that the graph only depends on the relative position of the points.

We are interested in the relationship between possible factor graphs on point processes of groups and various group theoretic properties. For example, on which groups are there point processes that admit factor graphs isomorphic to $\ZZ$? This was first investigated by Holroyd and Peres in \cite{holroyd2003}, where they prove that the \emph{Poisson point process} on $\RR^n$ admits such a factor graph for every $n$. This was later extended by Tim\'{a}r in \cite{timar2004} who proved that all free and ergodic point processes on $\RR^n$ admit such factor graphs, and moreover, admit factor graphs isomorphic to $\ZZ^d$ (with no dimensional restriction on $n$ and $d$). These papers were the inspiration for the present work -- in reading them one is struck by the similarity of the techniques with that of the theory of probability measure preserving countable Borel equivalence relations. This is no coincidence, and in leveraging that observation we are able to prove the maximum generalisation of those theorems:

\begin{thm}\label{amenabletheorem}
Let $G$ be a locally compact, second countable, unimodular, noncompact group and $\Pi$ a free and ergodic invariant point process on $G$ of finite intensity. 

Then $\Pi$ admits Cayley factor graphs of amenable discrete groups if and only if $G$ is amenable. In that case, it admits Cayley factor graphs of \emph{all} finitely generated infinite amenable groups.
\end{thm}

Here a \emph{Cayley factor graph} is a factor graph which happens to be the Cayley graph of some fixed countable group.

The strategy for proving Theorem \ref{amenabletheorem} is to rephrase the question in terms of an associated algebraic object. This algebraic object gives an alternative description of factor graphs and other factor constructions of interest, and its key features are summarised in the following theorem:

\begin{thm}\label{correspondencetheorem}
Let $G$ be a locally compact and second countable group, and $\Pi$ an invariant point process on $G$ with law $\mu$. 

Then associated to this data is an $r$-discrete probability measure preserving groupoid $(\Marrow, \muarrow)$ called \emph{the Palm groupoid} of $\Pi$. It has the following properties:
\begin{itemize}
    \item Factor thinnings $\theta : (\MM, \mu) \to \MM$ of $\Pi$ are in correspondence with Borel subsets $A$ of the unit space $\MMo$ of the Palm groupoid,
    \item Factor $\Xi$-markings $\mathscr{C} : (\MM, \mu) \to \Xi^\MM$ are in correspondence with Borel $\Xi$-valued maps $P$ defined on the unit space $\MMo$ of the Palm groupoid, and
    \item Factor graphs $\mathscr{G} : (\MM, \mu) \to \graph(G)$ of $\Pi$ are in correspondence with Borel subsets $\mathscr{A}$ of the arrow space $\Marrow$ of the Palm groupoid.
\end{itemize}

\end{thm}
 Once the above theorem is established, Theorem \ref{amenabletheorem} follows immediately from Ornstein-Weiss and Dye's Theorem. 

Having satisfactorily answered the question of Cayley factor graphs on amenable groups, we then turn to the opposite property -- Kazhdan's Property (T). Here we restrict ourselves to the study of a particular point process on such groups:

\begin{thm}\label{kazhdantheorem}
Let $G$ be a locally compact and second countable nondiscrete group with Kazhdan's Property (T), and $\Pi$ be the Poisson point process on $G$. Assume further that $G$ has no compact normal subgroups.

Then $\Pi$ admits \emph{no} Cayley factor graphs, and no factor of IID Cayley factor graphs.
\end{thm}

Here a \emph{factor IID graph} is an equivariantly defined graph where each point of the process is also allowed its own $\texttt{Unif}[0,1]$ random variable. 

In the case of discrete groups, the above theorem is due to Popa \cite{MR2342637} (see also Section 4 of \cite{vaes2006rigidity} for further discussion). Here references to the Poisson point process should be replaced by $\text{Ber}(p)$ percolation on the group, and the conclusion is that the only possible such factor graph is the Cayley graph of the group itself when $p = 1$.

Theorem \ref{kazhdantheorem} gives examples of probability measure preserving countable Borel equivalence relations which cannot be \emph{freely} generated by any action of a discrete group (the existence of such objects was first established by Furman in \cite{furman1999orbit}). The Poisson point process examples have additional properties, see Section \ref{tsection}.

\subsection*{Structure of paper}

In Section \ref{basicdefs} we set the scene and introduce point processes. This is meant as an overview for those unfamiliar with point processes and has no original content. For further details on the history of point processes and explicit proofs of technical facts one should consult \cite{vere2003introduction} and \cite{daley2007introduction}. A good gentle introduction to point process theory is \cite{MR1207584}. Two modern sources that explicitly discuss unimodularity in the context of point processes are \cite{baszczyszyn:cel-01654766} and \cite{MR3791470}.

In Section \ref{groupoid} we introduce the \emph{rerooting groupoid}, an object which governs the Borel factors of a point process. We also equip this groupoid with the \emph{Palm measure} of point processes and see how unimodularity of the ambient group manifests itself as the groupoid being probability measure preserving. In this way we see that not only \emph{Borel} factors are governed by the groupoid, but \emph{measured} ones as well. This is Theorem \ref{correspondencetheorem}.

In Section \ref{cayleyfactorgraphs} we apply the above theory to prove Theorem \ref{amenabletheorem} and Theorem \ref{kazhdantheorem}. 

Finally, in the appendix we include a discussion of \emph{cross-sections} and how they relate to point processes and Theorem \ref{amenabletheorem}.

\subsection*{Acknowledgements}
This work was partially supported by ERC Consolidator Grant 648017.   

This paper forms part of a thesis that the author wrote under the supervision of Mikl\'{o}s Ab\'{e}rt. Special thanks are given to Alessandro Carderi and Mikołaj Fraczyk for discussions on a preliminary version of this paper, and to Benjamin Hayes for suggesting a more general version of Theorem \ref{kazhdantheorem}.

 \tableofcontents

\section{Point processes and factors of interest}\label{basicdefs}

\subsection{Basic definitions}

Let $(Z,d)$ denote a complete and separable metric space (a csms). A \emph{point process on $Z$} is a random discrete subset of $Z$. We will also study random discrete subsets of $Z$ that are \emph{marked} by elements of an additional csms $\Xi$. Typically $\Xi$ will be a finite set that we think of as colours.

\begin{defn}
	The \emph{configuration space} of $Z$ is
    	\[
        	\MM(Z) = \{ \omega \subset Z \mid \omega \text{ is discrete} \},
        \]
    and the \emph{$\Xi$-marked configuration space} of $Z$ is
    	\[
        	\Xi^\MM(Z) = \{ \omega \subset Z \times \Xi \mid \omega \text{ is discrete, and if } (g, \xi) \in \omega \text{ and } (g, \xi') \in \omega \text{ then } \xi = \xi' \}.
        \]
\end{defn}

Note that $\Xi^\MM(Z) \subset \MM(Z \times \Xi)$. We think of a $\Xi$-marked configuration $\omega \in \Xi^\MM(Z)$ as a discrete subset of $Z$ with labels on each of the points (whereas a typical element of $\MM(Z \times \Xi)$ is a discrete subset where each point has possibly multiple marks). 

If $\omega \in \Xi^\MM(Z)$ is a marked configuration, then we will write $\omega_z$ for the unique element of $\Xi$ such that $(z, \omega_z) \in \omega$. 

The Borel structure on configuration spaces is exactly such that the following \emph{point counting functions} are measurable. Let $U \subseteq Z$ be a Borel set. It induces a function $N_U : \MM(Z) \to \NN_0 \cup \{ \infty \}$ given by
\[
	N_U(\omega) = \abs{\omega \cap U}.
\]

We will primarily be interested in point processes defined on locally compact and second countable (lcsc) groups $G$. Such groups admit a unique (up to scaling) Haar measure $\lambda$, we fix such a choice. Recall:

\begin{thm}[Struble's theorem, see Theorem 2.B.4 of \cite{MR3561300}]
Let $G$ be a locally compact topological group. Then $G$ is second countable \emph{if and only if} it admits a proper\footnote{Recall that a metric is \emph{proper} if closed balls are compact.} and left-invariant metric.
\end{thm}

Such a metric is unique up to coarse equivalence (bilipschitz if the group is compactly generated). We fix $d$ to be any such metric. 

\begin{thm}[See Theorem A2.6.III of \cite{vere2003introduction}]
	If $X$ is a complete and separable metric space, then $\MM(X)$ is a Borel subset of a complete and separable metric space\footnote{Here $\mathcal{M}^{\#}(X)$ denotes the space of locally finite Borel measures on $X$. It will not play a role in the present work, other than in witnessing that $\MM(X)$ is standard Borel.} $\mathcal{M}^{\#}(X)$, and is thus a standard Borel space.
\end{thm}

The above theorem implies that the theory of $\MM(X)$ valued random variables is well-behaved. 

We mostly consider the configuration space of a fixed group $G$. So out of notational convenience let us write $\MM = \MM(G)$ and $\Xi^\MM = \Xi^\MM(G)$. The latter here is an abuse of notation: formally $\Xi^\MM$ ought to denote the set of functions from $\MM$ to $\Xi$, but instead we are using it to denote the set of functions from \emph{elements} of $\MM$ to $\Xi$.

Note that the marked and unmarked configuration spaces of $G$ are Borel $G$-spaces. To spell this out, $G \acts \MM$ by $g \cdot \omega = g\omega$ and $G \acts \Xi^\MM$ by
\[
    g \cdot \omega = \{(gx, \xi) \in G \times \Xi \mid (g, \xi) \in \omega \}.
\]

\begin{defn}\label{ppreview}
	A \emph{point process} on $G$ is a $\MM(G)$-valued random variable $\Pi$, that is, a measurable function $\Pi: (\Omega, \PP) \to \MM(G)$, where $(\Omega, \PP)$ is some\footnote{As is usual in probability theory, the specifics of this probability space will never come up.} auxiliary probability space. Its \emph{law} or \emph{distribution} $\mu_\Pi$ is the pushforward probability measure $\Pi_*(\PP)$ on $\MM(G)$. It is \emph{invariant} if its law is an invariant probability measure for the action $G \acts \MM(G)$.
	
	The associated \emph{point process action} of an invariant point process $\Pi$ is $G \acts (\MM(G), \mu_\Pi)$.
\end{defn}

Some remarks and caveats are in order:
\begin{itemize}
	\item Point processes which are not invariant are very much of interest, but the only examples which we will consider will be so-called ``Palm point processes'', to be defined later. Thus unless explicitly prefaced by the word ``Palm'', one ought to interpret ``point process'' as ``invariant point process'' throughout this work.
	\item Sometimes one will see the term \emph{simply point process} for what we are calling point processes, as each point has multiplicity one. We simply use ``point process'' as we do not need higher multiplicity points in the present work.
    \item $\Xi$-marked point processes are defined similarly, with $\Xi^\MM$ taking the place of $\MM$. There isn't much difference between marked point processes and unmarked ones for our purposes (it's just a case of which is more convenient for the particular problem at hand). Thus ``point process'' might also mean ``marked point process''. This will also be reflected in definitions: if a concept is defined for point processes (and uses the symbol $\MM$), then it will also apply for marked point processes (using the symbols $\Xi^\MM$). 
    \item One could certainly define point processes on a discrete group, but this is better known as percolation theory. We are trying to move beyond that, so we will almost always implicitly assume $G$ is nondiscrete. 
    \item Another case of interest we will discuss  in a concurrently appearing work with the author and Mikl\'{o}s Ab\'{e}rt is $\Isom(S)$-invariant point processes on $S$, where $S$ is a Riemannian symmetric space. For instance, one would consider isometry invariant point processes on Euclidean space $\RR^n$ or hyperbolic space $\HH^n$. The general theory we introduce in this paper carries over to that context, and will be discussed in the other paper.
    \item Our interest in point processes is almost exclusively \emph{as actions}. We will therefore rarely distinguish between a point process proper and its distribution. Thus we may use expressions like ``suppose $\mu$ is a point process'' to mean ``suppose $\mu$ is the distribution of some point process''.
\end{itemize}

\begin{defn}
    
    The \emph{intensity} of a point process $\mu$ is
    \[
        \intensity(\mu) = \frac{1}{\lambda(U)} \EE_\mu \left[ N_U \right],
    \]
    where $U \subset G$ is any Borel set of positive (but finite) Haar measure, and $N_U(\omega) = \abs{\omega \cap U}$ is its point counting function.
\end{defn}

To see that the intensity is well-defined (that is, does not depend on our choice of $U$), observe that the function $U \mapsto \EE_\mu[N_U]$ defines a Borel measure on $G$ which inherits invariance from the shift invariance of $\mu$. So by uniqueness of Haar measure, it is some scaling of our fixed Haar measure $\lambda$ -- the intensity is exactly this multiplier. We also see that whilst the intensity depends on our choice of Haar measure, it scales linearly with it. \emph{We will almost exclusively concern ourselves with point processes of finite intensity.}

Note that a point process has intensity zero if and only if it is empty almost surely.

\subsection{Examples of point processes}

\begin{example}[Lattice shifts]

Let $\Gamma < G$ be a \emph{lattice}, that is, a discrete subgroup that admits an invariant probability measure $\nu$ for the action $G \acts G / \Gamma$. The natural map $\MM(G / \Gamma) \to \MM(G)$ given by
\[
    \omega \mapsto \bigcup_{a\Gamma \in \omega} a\Gamma
\]
is left-equivariant, and hence maps invariant point processes on $G / \Gamma$ to invariant point processes on $G$. In particular, we have the \emph{lattice shift}, given by choosing a $\nu$-random point $a\Gamma$.

\end{example}

\begin{example}[\textbf{Induction from a lattice}]
Now suppose one also has a pmp action $\Gamma \acts (X, \mu)$. It is possible to \emph{induce} this to a pmp action of $G$ on $G / \Gamma \times X$. This can be described as an $X$-marked point process on $G$. To do this, fix a fundamental domain $\mathscr{F} \subset G$ for $\Gamma$. Choose $f \in \mathscr{F}$ uniformly at random, and independently choose a $\mu$-random point $x \in X$. Let
	\[
	  \Pi = \{ (f\gamma, \gamma \cdot x) \in G \times X \mid \gamma \in \Gamma \}.  
	\]
	Then $\Pi$ is a $G$-invariant $X$-marked point process.
\end{example}

In this way one can view point processes as generalised lattice shift actions. Note that there are groups without lattices (for instance Neretin's group, see \cite{MR2881324}), but every group admits interesting point processes, as we discuss now. The most fundamental of these is known as \emph{the Poisson point process}. We will define this after reviewing the Poisson distribution:

Recall that a random integer $N$ is \emph{Poisson distributed with parameter $t > 0$} if
\[
\PP[N = k] = \exp(-t)\frac{t^k}{k!}.\]
We write $N \sim \texttt{Pois}(t)$ to denote this. It is convenient to extend this definition to $t = 0$ and $t = \infty$ by declaring $N \sim \texttt{Pois}(0)$ when $N = 0$ almost surely and $N \sim \texttt{Pois}(\infty)$ when $N = \infty$ almost surely.

\begin{defn}
	Let $X$ be a complete and separable metric space equipped with a non-atomic Borel measure $\lambda$.
	
	A point process $\Pi$ on $X$ is \emph{Poisson with intensity $t > 0$} if it satisfies the following two properties:
    	\begin{description}
        	\item[(Poisson point counts)] for all $U \subseteq G$ Borel, $N_U(\Pi)$ is a Poisson distributed random variable with parameter $t \lambda(U)$, and
            \item[(Total independence)] for all $U, V \subseteq G$ \emph{disjoint} Borel sets, the random variables $N_U(\Pi)$ and $N_V(\Pi)$ are \emph{independent}.
        \end{description}
\end{defn}

For reasons that should not be immediately apparent, both of the above defining properties are equivalent. We will write $\PPP_t$ for the distribution of such a random variable, or simply $\PPP$ if the intensity is understood. 

We think of the Poisson point process as a completely random scattering of points in the group. It is an analogue of Bernoulli site percolation for a continuous space.

We now construct the process (somewhat) explicitly. Partition $G$ into disjoint Borel sets $U_1, U_2, \ldots$ of positive but finite volume. For each of these, independently sample from a Poisson distribution with parameter $t \lambda(U_i)$. Place that number of points in the corresponding $U_i$ (independently and uniformly at random).

This description can be turned into an explicit sampling rule\footnote{That is, one can define a measurable function $f : \prod_n X_n \to \MM$ defined on an appropriate product of probability spaces such that the pushforward measure is the distribution of the Poisson point process.}, if one desires.

For proofs of basic properties of the Poisson point process (such as the fact that it does not depend on the partition chosen above), see the first five chapters of Kingman's book \cite{MR1207584}. 
\begin{defn}

A pmp action $G \acts (X, \mu)$ is \emph{ergodic} if for every $G$-invariant measurable subset $A \subseteq X$, we have $\mu(A) = 0$ or $\mu(A) = 1$.

The action is \emph{mixing} if for all measurable $A, B \subseteq (X, \mu)$ we have
\[
    \lim_{g \to \infty} \mu(gA \cap B) = \mu(A)\mu(B).
\]
The action is \emph{essentially free} if $\stab_G(x) = \{1\}$ for $\mu$ almost every $x \in X$. In the case of point process actions we will sometimes use the term \emph{aperiodic} to refer to this.

\end{defn}

\begin{prop}
	The Poisson point process actions $G \acts (\MM, \PPP_t)$ on a noncompact group $G$ are essentially free and ergodic (in fact, mixing).
\end{prop}

A proof of freeness that is readily adaptable to our setting can be found as Proposition 2.7 of \cite{MR3664810}. For ergodicity and mixing, see the proof of the discrete case in Proposition 7.3 of the Lyons-Peres book \cite{MR3616205}. It directly adapts, once one knows the required cylinder sets exist.

Although the subscript $t$ suggests that the Poisson point processes form a continuum family of actions, this is not always the case:

\begin{thm}[Ornstein-Weiss in \cite{ornstein1987entropy}, see also \cite{soo2019finitary}]
	Let $G$ be an amenable group which is not a countable union of compact subgroups. Then the Poisson point process actions $G \acts (\MM, \PPP_t)$ are all isomorphic.
\end{thm}

The following definition uses notation that does not appear in the literature (the object of course does, but there does not appear to be a symbolic representation for it):

\begin{defn}
	If $\Pi$ is a point process, then its \emph{IID version} is the $[0,1]$-marked point process $[0,1]^\Pi$ with the property that conditional on its set of points, its labels are independent and IID $\text{Unif}[0,1]$ distributed. If $\mu$ is the law of $\Pi$, then we will write $[0,1]^\mu$ for the law of $[0,1]^\Pi$.
	
	One can define the IID of a point process over spaces other than $[0,1]$ (for instance, $[n] = \{1, 2, \ldots, n\}$ with the counting measure), but we will only use the full IID.
	
\end{defn}

\begin{remark}

As we've mentioned, $[0,1]$-marked point processes on $G$ are particular examples of point processes on $G \times [0,1]$. One can show (see Theorem 5.6 of \cite{MR3791470}) that the Poisson point process on $G \times [0,1]$ with respect to the product measure $\lambda \otimes \text{Leb}$ is just the IID version of the Poisson point process on $G$, a fact which we will make use of later.

\end{remark}

\begin{prop}

The IID Poisson point process on a noncompact group is ergodic (and in fact mixing).

\end{prop}

This can be seen by viewing the IID Poisson on $G$ as the Poisson point process on $G \times S^1$, restricted to $G$. Note that the restriction of a mixing action to a noncompact subgroup is mixing.

\begin{remark}

One can define ``the IID'' of any probability measure preserving countable Borel equivalence relation, see \cite{MR3813200}. This construction is known as \emph{the Bernoulli extension}, and is ergodic if the base space is ergodic. 

\end{remark}

\begin{prop}

Let $\Pi$ be a point process on a group $G$ which is non-empty almost surely. Then $\abs{\Pi} = \infty$ almost surely if and only if $G$ is noncompact.

\end{prop}

\begin{proof}
    
    It is immediate that any point process on a compact group must be finite almost surely (as it is a discrete subset of the space).
    
    Now suppose $\Pi$ is a non-empty point process on $G$ which is finite almost surely. Then the IID of this process $[0,1]^\Pi$ still has this property. We define the following $G$-valued random variable:
    \[
        f([0,1]^\Pi) = \text{ the unique } x \in \Pi \text{ with maximal label in } [0,1]^\Pi.
    \]
    The invariance of the point process translates into equivariance of the map $f : [0,1]^\MM \to G$. Therefore the law of this random variable is an invariant \emph{probability} measure on $G$. Such a measure exists exactly when $G$ is compact.
\end{proof}

\subsection{Factors of point processes}

\begin{defn}
    
    A \emph{point process factor map} is a $G$-equivariant and measurable map $\Phi : \MM \to \MM$. If $\mu$ is a point process and $\Phi$ is only defined $\mu$ almost everywhere, then we will call it a \emph{$\mu$ factor map} or a \emph{factor of $\mu$}.
    
    We will be interested in two monotonicity conditions:
    \begin{itemize}
        \item if $\Phi(\omega) \subseteq \omega$ for all $\omega \in \MM$, we will call $\Phi$ a \emph{thinning} (and usually denote it by $\theta$), and 
        \item if $\Phi(\omega) \supseteq \omega$ for all $\omega \in \MM$, we will call $\Phi$ a \emph{thickening} (and usually denote it by $\Theta$).
    \end{itemize}
    
    We use the same terms for marked point processes as well. 
    
\end{defn}

\begin{remark}\label{thinningconfusio}

There are \emph{two} possible ways to interpret the above monotonicity conditions for a $\Xi$-marked point process, depending on what you want to do with the mark space. One can consider
\[
    \Phi : \Xi^\MM \to \Xi^\MM, \text{ or } \Phi : \Xi^\MM \to \MM.
\]  
In the former case, the definition above works verbatim. In the latter case, one should interpret a statement like ``$\omega \subseteq \Phi(\omega)$'' as ``$\omega$ is contained in the \emph{underlying set} $\pi(\Phi(\omega))$ of $\Phi(\omega)$, where $\pi : \Xi^\MM \to \MM$ is the map that forgets labels.
\end{remark}

The following example is implicit in the literature, but is not usually named and does not have a consistent symbolic representation. We will use it enough that we must name it:

\begin{example}[Metric thinning]\label{deltathinningdefn}
    
    Let $\delta > 0$ be a tolerance parameter. The \emph{$\delta$-thinning} is the equivariant map $\theta_\delta : \MM \to \MM$ given by
    \[
        \theta^\delta(\omega) = \{ g \in \omega \mid d(g, \omega \setminus \{g\} > \delta \}.
    \]
  
    When $\theta^\delta$ is applied to a point process, the result is always a $\delta$-separated point process (but possibly empty).
  
    We define $\theta^\delta$ in the same way for marked point processes (that is, it simply ignores the marks).
  
\end{example}

\begin{example}[Independent thinning]\label{independentthinning}
Let $\Pi$ be a point process. The \emph{independent $p$-thinning} defined on its IID $[0,1]^\Pi$ is given by
\[
    \mathcal{I}_p([0,1]^\Pi) = \{g \in \Pi \mid \Pi_g \leq p \}.
\]
\end{example}

One can show that independent $p$-thinning of the Poisson point process of intensity $t > 0$ yields the Poisson point process of intensity $pt$, as one would expect. See Section 5.3 of \cite{MR3791470} for further details.

\begin{example}[Constant thickening]\label{constantthickening}
    
    Let $F \subset G$ be a finite set containing the identity $0 \in G$, and $\Pi$ be a point process which is \emph{$F$-separated} in the sense that $\Pi \cap \Pi f = \empt$ for all $f \in F\setminus\{0\}$. Then there is the associated thickening $\Theta^F(\Pi) = \Pi F$. 
    It is intuitively obvious that $\intensity (\Theta^F(\Pi)) = \abs{F} \intensity (\Pi)$. This can be formally established as follows: let $U \subseteq G$ be of unit volume. Then
    \begin{align*}
        \intensity (\Theta^F(\Pi) ) &= \EE[\abs{U \cap \Pi F} ] && \text{by definition}  \\
        &= \sum_{f \in F} \EE[\abs{U \cap \Pi f}] && \text{by }F\text{-separation} \\
        &= \sum_{f \in F} \EE[ \abs{Uf^{-1} \cap \Pi} ] && \\
        &= \sum_{f \in F} \EE[\abs{U \cap \Pi}] && \text{by \emph{unimodularity}} \\
        &= \abs{F} \intensity (\Pi).
    \end{align*}
    This is the first real appearance of our unimodularity assumption. 
    
    In particular, we can demonstrate that $\intensity \Theta^F(\Pi) = \abs{F} \intensity \Pi$ is \emph{not} automatically true without unimodularity. For this, let $\Pi$ denote the unit intensity Poisson point process on $G$, and $F = \{0, f\}$ where $f \in G$ is chosen such that $\lambda(Uf^{-1}) < 1$. Then $\abs{Uf^{-1} \cap \Pi}$ is Poisson distributed with parameter $\lambda(Uf^{-1})$, and so by the above calculation $\intensity\Theta^F(\Pi) < 2 \cdot \intensity \Pi$.
\end{example}

Monotone maps have been investigated in the specific case of the Poisson point process on $\RR^n$. We note the following interesting theorems:

\begin{thm}[Holroyd, Peres, Soo \cite{MR2884878}]
Let $s > t > 0$. Then the Poisson point process on $\RR^n$ of intensity $s$ can be thinned to the Poisson point process of intensity $t$. That is, there exists an equivariant and deterministic map $\theta : (\MM(\RR), \PPP_s) \to (\MM(\RR), \PPP_t)$.

\end{thm}

\begin{thm}[Gurel-Gurevich and Peled \cite{MR3096589}]
Let $s > t > 0$ be intensities. Then the Poisson point process on $\RR^n$ of intensity $s$ cannot be thickened to the Poisson point process of intensity $t$. That is, there is no equivariant and deterministic map $\Theta : (\MM(\RR), \PPP_s) \to (\MM(\RR), \PPP_t)$.

\end{thm}

We stress in the above theorems the \emph{deterministic} nature of the maps. If one is allowed additional randomness (that is, one asks for a factor of IID map), then both theorems are easily established. In fact, the IID of \emph{any} point process factors onto the Poisson of arbitrary intensity.

\begin{defn}
    
    A \emph{factor $\Xi$-marking} of a point process is a $G$-equivariant map $\mathscr{C} : \MM \to \Xi^\MM$ such that the underlying subset in $G$ of $\mathscr{C}(\omega)$ is $\omega$. That is, $\mathscr{C}$ is a rule that assigns a mark from $\Xi$ to each point of $\omega$ in some deterministic way. Again, if $\mathscr{C}$ is only defined $\mu$ almost everywhere then we will call it a \emph{$\mu$ factor $\Xi$-marking}.
    
    We will also use the term ``colouring'' for the same thing.
    
\end{defn}

\begin{example}
    
    Let $\theta : \MM \to \MM$ be a thinning. Then the associated $2$-colouring is $\mathscr{C}_\theta : \MM \to \{0, 1\}^\MM$ given by
    \[
        \mathscr{C}_\theta(\omega) = \{ (g, \1_{g \in \theta(\omega)} \in G \times \{0, 1\} \mid g \in \omega \}.
    \]
    We will see that all markings are built out of thinnings in a similar way.
\end{example}

\begin{remark}\label{thinninglost}

There is a difference between the \emph{thinning map $\theta$} and the resulting \emph{thinned process $\theta_*(\mu)$} that can be a source for confusion. Passing to the thinned process (in principle) can lose information about $\mu$.

For example, let $\Pi$ denote a Poisson point process on $G$ and $\Upsilon$ an independent random shift of a lattice $\Gamma < G$. Define the following thinning $\theta : \MM \to \MM$ by
\[
    \theta(\omega) = \{ g \in \omega \mid g\Gamma \subseteq \omega \}.
\]
Then $\theta(\Pi \cup \Upsilon) = \Upsilon$, and so the thinning completely loses the Poisson point process.

\end{remark}

\begin{defn}\label{inputoutputdefn}
Let $\Phi : \MM \to \MM$ be a factor map. We think of its input $\omega$ as being red, its output $\Phi(\omega)$ as being blue, and their overlap $\omega \cap \Phi(\omega)$ as being purple. 

For $g \in \omega$, let $\texttt{Colour}(g) \in \{\text{Red, Blue, Purple}\}$ be
\[ \texttt{Colour}(g) = 
    \begin{cases}
        \text{Red} & \text{ If } g \in \omega \setminus \Phi(\omega), \\
        \text{Blue} & \text{ If } g \in \Phi(\omega)\setminus \omega, \\
        \text{Purple} & \text{ If } g \in \omega \cap \Phi(\omega).
    \end{cases}
\]
Now define $\Theta^\Phi : \MM \to \{\text{Red, Blue, Purple}\}^\MM$ to be the following \emph{input/output thickening} of $\Phi$ defined by
\[
    \Theta^\Phi(\omega) = \{ (g, \texttt{Colour}(g)) \in G \times \text{Red, Blue, Purple}\} \mid g \in \omega \}.
\]

Let $\pi : \{\text{Red, Blue, Purple}\}^\MM \to \MM$ be the projection map that deletes red points and then forgets colours, that is,
\[
    \pi(\omega) = \{ g \in \omega \mid \omega_g \in \{\text{Blue, Purple}\} \}.
\]

\end{defn}

\begin{remark}\label{factorsdecompose}
Observe that $\Phi = \pi \circ \Theta^\Phi$ -- that is, an \emph{arbitrary} factor map decomposes as the composition of a thinning and a thickening. In this way we can often reduce the study of arbitrary factors to that of \emph{monotone} factors.
\end{remark}

\begin{defn}
	The \emph{space of graphs in $G$} is
    \[
    	\graph(G) = \{ (V, E) \in \MM(G) \times \MM(G \times G) \mid E \subseteq V \times V \}.
    \]
    This is a Borel $G$-space (with the diagonal action). 
    
    A \emph{factor graph} is a measurable and $G$-equivariant map $\Phi : \MM(G) \to \graph(G)$ with the property that the vertex set of $\Phi(\omega)$ is $\omega$.
    
    If a factor graph is connected, then we will refer to it as a \emph{graphing}.
    
\end{defn}

\begin{remark}
	The elements of $\graph(G)$ are technically directed graphs, possibly with loops, and without multiple edges between the same pair of vertices. It's possible to define (in a Borel way) whatever space of graphs one desires (coloured, undirected, etc.) by taking appropriate subsets of products of configuration spaces.
\end{remark}

\begin{remark}
	One might prefer to call factor graphs as above \emph{monotone} factor graphs, as they never modify the vertex set. Our terminology follows that of probabilists, see for instance \cite{holroyd2005}. We have not yet found a use for the less restrictive factor graph concept. 
\end{remark}

\begin{example}\label{distanceR}
    
    The \emph{distance-$R$} factor graph is the map $\mathscr{D}_R : \MM \to \graph(G)$ given by
    \[
        \mathscr{D}_R(\omega) = \{ (g, h) \in \omega \times \omega \mid d(g, h) \leq R \}.
    \]
    The connectivity properties of this graph fall under the purview of continuum percolation theory, see for instance \cite{meester1996continuum}.
\end{example}

\section{The rerooting equivalence relation and groupoid}\label{groupoid}

We now introduce a pair of algebraic objects that capture factors of a point process. For exposition's sake, we will first discuss unmarked point processes on a group $G$. It is assumed that the reader is somewhat familiar with the notion of a probability measure preserving (pmp) countable Borel equivalence relation (cber), and has heard the definition of a groupoid (but no more knowledge is required than that). For more information on pmp cbers see \cite{kechris2004topics} and \cite{gaboriau2016around}.

\begin{defn}
    The \emph{space of rooted configurations on $G$} is
    \[
        \MMo(G) = \{ \omega \in \MM(G) \mid 0 \in \omega \}.
    \]
    
    If $G$ is understood, then we will drop it from the notation for clarity.
    
    The \emph{rerooting equivalence relation} on $\MMo$ is the orbit equivalence relation of $G \acts \MM$ restricted to $\MMo$. Explicitly: 
    \[
        \Rel = \{ (\omega, g^{-1}\omega) \in \MMo \times \MMo \mid g \in \omega \}.
    \]
    
    This defines a countable Borel equivalence relation structure on $\MMo$. It is degenerate whenever $\omega \in \MMo$ exhibits symmetries: for instance, the equivalence class of $\ZZ$ viewed as an element of $\MMo(\RR)$ is a singleton. We are usually interested in essentially free actions, where such difficulties will not occur. Nevertheless, we do care about lattice shift point processes and so we will introduce a groupoid structure that keeps track of symmetries.
    
    The \emph{space of birooted configurations} is
    \[
        \Marrow = \{ (\omega, g) \in \MMo \times G \mid g \in \omega \}.
    \]
    
    We visualise an element $(\omega, g) \in \Marrow$ as the rooted configuration $\omega \in \MMo$ with an arrow pointing to $g \in \omega$ from the root (ie, the identity element of $G$).
    
    The above spaces form a \emph{groupoid} $(\MMo, \Marrow)$ which we will refer to as the \emph{rerooting groupoid}. Its unit space is $\MMo$ and its arrow space is $\Marrow$. We can identify $\MMo$ with $\MMo \times \{0\} \subset \Marrow$. 
    
    The multiplication structure is as follows: we declare a pair of birooted configurations $(\omega, g), (\omega', h)$ in $\Marrow$ to be \emph{composable} if $\omega' = g^{-1}\omega$, in which case
    \[
        (\omega, g) \cdot (\omega', h) := (\omega, gh).
    \]
    
    Note that if $\Gamma < G$ is a discrete subgroup (so $\Gamma \in \MMo(G)$), then the above multiplication is just the usual one.
    
    The \emph{source map} $s : \Marrow \to \MMo$ and \emph{target map} $t : \Marrow \to \MMo$ are
    \[
        s(\omega, g) = \omega, \text{ and } t(\omega, g) = g^{-1}\omega.
    \]
\end{defn}

    Note that the rerooting groupoid is \emph{discrete} in the sense that $s^{-1}(\omega)$ is at most countable for all $\omega \in \MMo$. 

\begin{remark}\label{aperiodic}

Let $\Maper_0$ denote the set of rooted configurations $\omega$ that are \emph{aperiodic} in the sense that $\stab_G(\omega) = \{e\}$. Then for all $\omega' \in \MMo$, there is \emph{at most} one $g \in G$ such that $g^{-1}\omega = \omega'$. Groupoids with this property are called \emph{principal}, and their groupoid structure is simply that of an equivalence relation (with a unique arrow between related points, and no arrows between unrelated points). 

Note also that the rerooting equivalence class $[\omega]_\Rel$ of such an aperiodic configuration $\omega$ is naturally \emph{parametrised} by $\omega$ itself: that is, the map
\begin{align*}
    &p_\omega : \omega \to [\omega]_\Rel \\
    &p_\omega(g) = g^{-1}\omega
\end{align*}
is a bijection.
\end{remark}

\begin{defn}

If $\Xi$ is a space of marks, then the \emph{space of $\Xi$-marked rooted configurations} is 
\[
    \Xi^\MMo = \{ \omega \in \Xi^\MM \mid \exists \xi \in \Xi \text{ such that } (0, \xi) \in \omega \}.
\]

The \emph{$\Xi$-marked rerooting groupoid} is defined as previously, with $\Xi^\MMo$ taking the place of $\MMo$.

\end{defn}

\subsection{Borel correspondences between the groupoid and factors}\label{borelcorrespondences}

With the definition of the rerooting groupoid in hand, we are now able to prove the Borel version of Theorem \ref{correspondencetheorem}.

Suppose $\theta : \MM \to \MM$ is an equivariant and measurable thinning. Then we can associate to it a subset of the rerooting groupoid, namely
\[
    A_\theta = \{ \omega \in \MM \mid 0 \in \theta(\omega) \}.
\]
This association has an inverse: given a Borel subset $A \subseteq \MMo$, we can define a thinning $\theta^A : \MM \to \MM$
\[
    \theta^A(\omega) = \{g \in \omega \mid g^{-1}\omega \in A \}.
\]

Thus we see that \emph{Borel subsets $A \subseteq \MMo$ of the rerooting groupoid correspond to Borel thinning maps $\theta : \MM \to \MM$}.

In the $\Xi$-marked case, one associates to a subset $A \subseteq \Xi^\MMo$ a thinning $\theta^A : \Xi^\MM \to \Xi^\MM$. 

In a similar way, we can see that if $P : \MMo \to [d]$ is a Borel partition of $\MMo$ into $d$ classes, then there is an associated factor $[d]$-colouring $\mathscr{C}^P : \MM \to [d]^\MM$ given by
\[
    \mathscr{C}^P(\omega) = \{ (g, P(g^{-1}\omega) \in G \times [d] \mid g \in \omega \},
\]
and given a factor $[d]$-colouring $\mathscr{C} : \MM \to [d]^\MM$ one associates the partition $P^\mathscr{C} : \MMo \to [d]$ given by
\[
    P(\omega) = c, \text{ where } c \text{ is the unique element of } [d] \text{ such that } (0, c) \in \mathscr{C}(\omega).  
\]
Again, these associations are mutual inverses. 

More generally, we see that \emph{Borel factor $\Xi$-markings $\mathscr{C} : \MM \to \Xi^\MM$ correspond to Borel maps $P : \MMo \to \Xi$}.

Now suppose that $\mathscr{G} : \MM \to \graph(G)$ is an equivariant and measurable factor graph. Then we can associate to it a subset of the rerooting groupoid's arrow space, namely
\[
    \mathscr{A}_\mathscr{G} = \{ (\omega, g) \in \Marrow \mid (0,g) \in \mathscr{G}(\omega)\}.
\]
In the other direction, we associate to a subset $\mathscr{A} \subseteq \Marrow$ the factor graph $\mathscr{G}^\mathscr{A} : \MM \to \graph(G)$
\[
    \mathscr{G}^\mathscr{A}(\omega) = \{ (g, h) \in \omega \times \omega \mid (g^{-1}\omega, g^{-1}h) \in \mathscr{A} \}.
\]

Thus we see that \emph{Borel subsets $\mathscr{A} \subseteq \Marrow$ of the rerooting groupoid's arrow space correspond to Borel factor graphs $\mathscr{G} : \MM \to \graph(G)$}.

Note also that the factor graph $\mathscr{G}$ is \emph{connected} for every input $\omega$ if and only if the corresponding subset $\mathscr{A}_\mathscr{G} \subseteq \Marrow$ generates the rerooting groupoid.

\begin{remark}

If $\mu$ is a point process, then the correspondence still works in one direction: namely, we can associate subsets $A \subset \MMo$ (or $\mathscr{A} \subseteq \Marrow$) to $\mu$-thinnings $\theta^A: (\MM, \mu) \to \MM$ (or $\mu$-factor graphs $\mathscr{G}_\mathscr{A}: (\MM, \mu) \to \MM$ respectively). 

We run into trouble in the other direction: suppose $\theta : \MM \to \MM$ is a thinning, but only defined $\mu$ almost everywhere. We wish to restrict it to $\MMo$, but a priori this makes no sense -- that is a subset of measure zero. It turns out that there is a way to make sense of this due to equivariance, but it will require some more theory that we explain in the next section.

\end{remark}
\subsection{The Palm measure}

We will now associate to a (finite intensity) point process $\mu$ a probability measure $\mu_0$ defined on the rerooting groupoid $\MMo$. When the ambient space is unimodular, this will turn the rerooting groupoid into a \emph{probability measure preserving (pmp) discrete groupoid}.

Informally, the Palm measure of a point process $\Pi$ is the process conditioned to contain the root. A priori this makes no sense (the subset $\MMo$ has probability zero), but there is an obvious way one could interpret the statement: condition on the process to contain a point in an $\e$ ball about the root, and take the limit as $\e$ goes to zero. See Theorem 13.3.IV of \cite{daley2007introduction} and Section 9.3 of \cite{MR3791470} for further details.

We will instead take the following concept of \emph{relative rates} as our basic definition:

\begin{defn}
	Let $\Pi$ be a point process of finite intensity with law $\mu$. Its (normalised) \emph{Palm measure} is the probability measure $\mu_0$ defined on Borel subsets of $\MMo$ by
    \[
    	\mu_0(A) := \frac{\intensity(\theta^A(\Pi))}{\intensity(\Pi)},
    \]
    where $\theta^A$ is the thinning associated to $A \subseteq \MMo$. 
    
    More explicitly,
    \[
    	\mu_0(A) := \frac{1}{\intensity(\mu)} \EE_\mu \left[ \#\{g \in U \mid g^{-1}\omega \in A \} \right],
    \]
    where $U \subseteq G$ is of unit volume.

We also define the Palm measure of a $\Xi$-marked point process similarly, with $\Xi^\MMo$ taking the place of $\MMo$.

A \emph{Palm version} of $\Pi$ is any random variable $\Pi_0$ with law $\mu_0$. That is, if for all Borel $B \subseteq \MMo$ we have
\[
    \PP[\Pi_0 \in B] = \mu_0(B).
\]

\end{defn}

We now describe some \emph{Palm calculus} -- that is, how the operation of ``take the Palm measure of'' behaves with respect to various factor operations.  This is just to build intuition for readers unfamiliar with the Palm measure. The proofs follow from an elementary symbolic manipulation of the definitions, so we omit them in the present work, and they will appear in a concurrent work with the author and Mikl\'{o}s Ab\'{e}rt.

\begin{example}[Forgetting labels]
If $\Pi$ is a \emph{labelled} point process, then the Palm measure of $\Pi$ \emph{after} we forget the labels is the same thing as forgetting the labels from the Palm measure $\Pi_0$.

More explicitly, if $\pi : \Xi^\MM \to \MM$ is the map that forgets labels, then $\pi(\Pi)_0$ has the same distribution as $\pi(\Pi_0)$.
\end{example}

\begin{example}[Lattice actions]
	If $\Gamma < G$ is a lattice, then the Palm measure of the associated lattice shift is just $\delta_\Gamma$ -- that is, the atomic measure on $\Gamma \in \MMo(G)$. 
    More generally, if $\Gamma \acts (X, \mu)$ is a pmp action, then the Palm measure of the associated induced $X$-marked point process is its \emph{symbolic dynamics}. That is, the map $\Sigma : (X, \mu) \to X^\MM$ given by
    \[
        \Sigma(x) = \{ (\gamma, \gamma^{-1} \cdot x) \in G \times X \mid \gamma \in \Gamma \}.
    \]
    pushes forward $\mu$ to the Palm measure. In words, you sample a $\mu$-random point $x \in X$ and track its orbit under $\Gamma$ (the inverse is an artefact of our left bias).
\end{example}

\begin{remark}

Suppose $\Pi$ is a finite intensity point process such that its Palm version is an atomic measure, say $\Pi_0 = \Omega$ almost surely where $\Omega \in \MMo$. Then $\Omega$ is a lattice in $G$. Note that $\Omega$ is automatically a discrete subset of $G$, and a simple mass transport argument shows that it is a subgroup. The covolume of this subgroup is the reciprocal of the intensity of $\Pi$.

\end{remark}

\begin{example}[Mecke-Slivnyak Theorem]\label{palmofpoisson}
    
    If $\Pi$ is a Poisson point process, then its Palm measure has the same law as $\Pi \cup \{0\}$, where $0 \in G$ is the identity. 
    
    In fact, this is a \emph{characterisation} of the Poisson point process: if the Palm measure of $\mu$ is obtained by simply adding the root\footnote{More formally, consider the map $F : \MM \to \MMo$ given by $F(\omega) = \omega \cup \{0\}$, by ``adding the root'' we mean the Palm measure $\mu_0$ is the pushforward $F_*\mu$.}, then $\mu$ is the Poisson point process (of some intensity).
    
\end{example}

The proof of the above fact can be found in Section 9.2 of \cite{MR3791470}. As a consequence, the Palm measure of the IID Poisson is the IID of the Palm measure of the Poisson itself. 

\begin{example}[Thinnings]
    The Palm version $\theta(\Pi)_0$ of a thinning $\theta = \theta^A$ of $\Pi$ (determined by a subset $A \subseteq \MMo$) is described in terms of its Palm version $\Pi_0$ as a conditional probability as follows:
    \[
        \PP[\theta(\Pi)_0 \in \bullet] = \PP[\theta(\Pi_0) \in \bullet \mid \Pi_0 \in A]
    \]
    for any $B \subseteq \MMo$.
    
    That is, the Palm measure $\theta(\Pi)_0$ can be obtained by sampling from $\Pi_0$ conditioned that the root is retained in the thinning, and then applying the thinning.
\end{example}

\begin{example}[Colourings]\label{palmofcolouring}
    The Palm version $\mathscr{C}(\Pi)_0$ of a $2$-colouring $\mathscr{C} : \MM \to \{0,1\}^\MM$ is simply $\mathscr{C}(\Pi_0)$.

\end{example}

\begin{example}\label{palmofthickening}
    
    Let $\Theta = \Theta^F$ be a constant thickening determined by $F \subset G$, as described in Example \ref{constantthickening}. If $\Pi$ is an $F$-separated process, then the Palm version $\Theta(\Pi)_0$ of the thickening $\Theta(\Pi)$ is as follows: sample from $\Pi_0$, and independently choose to root $\Theta(\Pi_0)$ at a uniformly chosen element $X$ of $F$. That is, $\Theta(\Pi)_0 \overset{d}{=} X^{-1} \Theta(\Pi_0)$.
    
    To see this, we compute\footnote{When we define the Palm measure of a set $B \subseteq \MMo$, we usually write ``$g \in U$'' rather than ``$g \in U \cap \Pi$'', as the latter condition $g^{-1} \Pi \in B$ already implies $g \in \Pi$. For this computation it is better to really spell it out though.} as follows:
    
    \begin{align*}
        &\PP[\Theta(\Pi)_0 \in B] = \frac{1}{\intensity \Theta(\Pi)} \EE[ \#\{g \in U \cap \Pi F \mid g^{-1} \Theta(\Pi) \in B \} ] && \text{By definition} \\
        &= \frac{1}{\abs{F}} \frac{1}{\intensity \mu} \sum_{f \in F} \EE[ \#\{g \in U \cap \Pi f \mid g^{-1} \Theta(\Pi) \in B \} ] && \text{By Example \ref{constantthickening}} \\
         &= \frac{1}{\abs{F}} \frac{1}{\intensity \mu} \sum_{f \in F} \EE[ \#\{g \in Uf^{-1} \cap \Pi \mid g^{-1} \Pi \in \Theta^{-1}(B) \} ] && \text{By equivariance} \\
         &= \frac{1}{\abs{F}} \frac{1}{\intensity \mu} \sum_{f \in F} \EE[ \#\{g \in U \cap \Pi \mid g^{-1} \Pi \in \Theta^{-1}(B) \} ] && \text{By unimodularity} \\
         &= \frac{1}{\abs{F}} \sum_{f \in F} \PP[ \Pi_0 \in \Theta^{-1}(B)] && \text{By definition} \\
         &= \frac{1}{\abs{F}} \sum_{f \in F} \PP[ \Theta(\Pi_0) \in B] && \\
         &= \PP[X^{-1} \Theta(\Pi_0) \in B].
    \end{align*}
    
\end{example}

\begin{remark}

A similar formula holds for \emph{arbitrary} thickenings, however one must size-bias in an appropriate way.

\end{remark}

The Palm measure has an associated integral equation. One writes
\[
(\lambda \otimes \mu_0)(U \times A) = \int_G \EE_0[\1_{U \times A}] d\lambda(x)\]
and then invokes the usual voodoo to extend a statement about measurable sets to one about measurable functions. We follow the terminology of \cite{MR3791470} by referring to the resulting formula as ``the CLMM'', and use it to prove the Mass Transport Principle:

\begin{thm}[Campbell-Little-Mecke-Matthes]\label{CLMM}
	Let $\mu$ be a finite intensity point process on $G$ with Palm measure $\mu_0$. Write $\EE$ and $\EE_0$ for the associated integral operators.
    
    If $f : G \times \MMo \to \RR_{\geq 0}$ is a measurable function (\emph{not} necessarily invariant in any way), then
    \[
    	\EE \left[\sum_{x \in \omega} f(x, x^{-1}\omega) \right] = \intensity(\mu) \EE_0 \left[ \int_G   f(x, \omega)  d\lambda(x)\right].
    \]
\end{thm}

Note that summing against $\omega$ is the same as integrating $G$ against $\omega$ viewed as a locally finite measure on $G$. 

\begin{remark}\label{VIF}

If $\nu$ is a point process with $\nu_0 = \mu_0$, then $\nu = \mu$, that is, the Palm measure \emph{determines} the point process.

To see this, we use the existance of a map $\mathscr{V} : [0,1] \times \MMo \to \MM$ with the property that if $\mu$ is \emph{any} point process with Palm measure $\mu_0$, then $\mathscr{V}_*(\text{Leb} \otimes \mu_0) = \mu$. This is a consequence of the \emph{Voronoi inversion formula}, see Section 9.4 of \cite{MR3791470}. 

\end{remark}

\subsection{Unimodularity and the Mass Transport Principle}\label{unimodularity}

The source and range maps $s, t : \Marrow \to \MM$ induce a pair of measures on $\Marrow$ defined by
\[
    \muarrow^s(\mathscr{G}) = \int_\MMo \abs{s^{-1}(\omega) \cap \mathscr{G}(\omega)} d\mu_0(\omega), \text{ and } \muarrow^t(\mathscr{G}(\omega)) = \int_\MMo \abs{t^{-1}(\omega) \cap \mathscr{G}} d\mu_0(\omega).
\]
In our factor graph interpretation this corresponds to the expected indegree and outdegree of $\mathscr{G}$ respectively, where we view $\mathscr{G}$ as a \emph{directed} graph. To see this, recall that for a rooted configuration $\omega \in \MMo$,
\[
    s^{-1}(\omega) = \{(\omega, g) \in \MMo \times G \mid g \in \omega\} \text{ and } t^{-1}(\omega) = \{(g^{-1}\omega, g^{-1}) \in \MMo \times G \mid g \in \omega \},
\]
and that there is an edge from $0$ to $g$ in $\mathscr{G}(\omega)$ exactly when $(\omega, g) \in \mathscr{G}$, and an edge from $g$ to $0$ exactly when $(g^{-1}\omega, g^{-1}) \in \mathscr{G}$. Thus
\[
    \overrightarrow{\deg}_0({\mathscr{G}(\omega)}) = \abs{s^{-1}(\omega) \cap \mathscr{G}(\omega)} \text{ and } \overleftarrow{\deg}_0({\mathscr{G}(\omega)}) = \abs{t^{-1}(\omega) \cap \mathscr{G}(\omega)}.
\]

\begin{remark}

We have had to adapt notation to suit our purposes. Usually a groupoid would be denoted by a letter like $\mathcal{G}$, and that is the set of arrows. Then its units would be denoted $\mathcal{G}_0$. We have tried to match this up with the necessary notation from point process theory as closely as possible. 

We choose to denote outdegree by an expression like $\overrightarrow{\deg}_0({\mathscr{G}(\omega)})$ instead of $\deg^+_{\mathscr{G}(\omega)}(0)$ as the arrows are more evocative, and the subscript notation becomes very small (as in, for instance, $\deg^+_{\mathscr{G}(\Pi_0)}(0)$. 

\end{remark}

\begin{prop}\label{pmpgroupoid}
    
    If $G$ is \emph{unimodular}, then $\muarrow^s = \muarrow^t$. That is, $(\Marrow, \muarrow)$ forms a discrete pmp groupoid.
    
    Equivalently, if $\Pi_0$ is the Palm version of any point process $\Pi$ on $G$, then
    \[
        \EE\left[ \overrightarrow{\deg}_0({\mathscr{G}(\Pi_0)}) \right] = \EE\left[ \overleftarrow{\deg}_0({\mathscr{G}(\Pi_0)}) \right].
    \]
    
    We will denote by $\muarrow$ this common measure $\muarrow^s = \muarrow^t$.
\end{prop}

\begin{proof}[Proof of Proposition \ref{pmpgroupoid}]
    
    \begin{align*}
        \muarrow^s(\mathscr{G}) &= \EE_{\mu_0} \left[ \sum_{g \in \omega} \1_{(\omega, g) \in \mathscr{G}} \right] && \text{by definition} \\
        &= \EE_{\mu_0} \left[\int_G \1_{x \in U} \sum_{g \in \omega} \1_{(\omega, g) \in \mathscr{G}} d\lambda(x) \right] && \text{For any } U \subseteq G \text{ of unit volume} \\
        &= \frac{1}{\intensity \mu} \EE_\mu \left[ \sum_{x \in \omega} \1_{x \in U} \sum_{g \in x^{-1}\omega} \1_{(x^{-1}\omega, g) \in \mathscr{G}}   \right] && \text{By the CLLM}    \\
        &= \frac{1}{\intensity \mu} \EE_\mu \left[ \sum_{h \in \omega}  \sum_{hg^{-1} \in \omega} \1_{hg^{-1} \in U} \1_{(gh^{-1}\omega, g) \in \mathscr{G}}   \right] && \text{Fubini and variable change } h = xg  \\
        &= \EE_{\mu_0} \left[ \int_G \sum_{g \in \omega} \1_{h^{-1}g \in U} \1_{(g\omega, g) \in \mathscr{G}} d\lambda(h) \right] && \text{By the CLLM}  \\
       &= \EE_{\mu_0} \left[ \sum_{g \in \omega}  \underbrace{\left( \int_G \1_{h^{-1}g \in U} d\lambda(h) \right)}_{= \lambda((Ug)^{-1})}   \1_{(g\omega, g) \in \mathscr{G}}  \right] && \text{Fubini}    \\
       &= \EE_{\mu_0} \left[ \sum_{g \in \omega}    \1_{(g\omega, g) \in \mathscr{G}}  \right] && \text{By unimodularity}  \\
       &= \muarrow^t(\mathscr{G}).
    \end{align*}
\end{proof}

\begin{defn}

The \emph{Palm groupoid} of a point process $\Pi$ with law $\mu$ is $(\Marrow, \muarrow)$. If $\Pi$ is free, then this groupoid is principal, and thus we refer to $\Pi$'s \emph{Palm equivalence relation} $(\MMo, \Rel, \mu_0)$.

\end{defn}

\begin{remark}

To the author's knowledge, the only direct references in the literature to the existence of this equivalence relation can be found in a paper of Avni \cite{avni2005spectral} (Example 2.2) and a paper of Bowen \cite{bowen2018all} (Questions and comments, item 1). 

There are also implicit references: see \cite{daley2007introduction}, \cite{murphy2017point}, \cite{MR3771756}.

\end{remark}

\begin{remark}
At this point one may be wondering what to do about the \emph{cost} (in the sense of Levitt and Gaboriau) of the above pmp equivalence relation. The author and Mikl\'{o}s Ab\'{e}rt explore this topic in a concurrently appearing work, where it is shown for example that the Poisson point process action has \emph{maximal} cost amongst all free actions of a group.
\end{remark}

By the usual routine for extending a statement about equality of measures to equality of integrals one can deduce from Proposition \ref{pmpgroupoid} \emph{The Mass Transport Principle}:

\begin{thm}[The Mass Transport Principle]
    Let $\mu$ be a point process on a unimodular group. Suppose $T : G \times G \times \MM \to \RR_{\geq 0}$ is a measurable function which is \emph{diagonally invariant} in the sense that $T(gx, gy; g\omega) = T(x, y; \omega)$ for all $g \in G$. Then
    \[
        \EE_{\mu_0} \left[ \sum_{x \in \omega} T(x, 0; \omega) \right] = \EE_{\mu_0} \left[ \sum_{y \in \omega} T(0, y; \omega) \right].
    \]
\end{thm}

We view $T(x, y; \omega)$ as representing an amount of \emph{mass} sent from $x$ to $y$ when the configuration is $\omega$. Thus the integrand on the lefthand side represents the total mass received from the root, and similarly the integrand on the righthand side represents the total mass sent from the root.

The mass transport principle immediately follows from Proposition \ref{pmpgroupoid}, as it just represents the integral of the function $\omega \mapsto \sum_{x \in \omega} T(x, 0; \omega)$ with respect to $\muarrow^t$ and $\muarrow^s$.

\subsection{Ergodicity and the factor correspondences in the measured category}\label{ergodicity}

\begin{defn}
    A subset $A \subseteq \MM$ of unrooted configurations is \emph{shift-invariant} if for all $\omega \in A$ and $g \in G$, we have $g\omega \in A$.

	A subset $A_0 \subseteq \MMo$ of rooted configurations is \emph{rootshift invariant} if for all $\omega \in A_0$ and $g \in \omega$, we have $g^{-1}\omega \in A_0$. 
	
	The groupoid $(\Marrow, \muarrow)$ is \emph{ergodic} if every rootshift invariant subset $A \subseteq \MMo$ has $\mu_0(A) = 0$ or $1$.
	
\end{defn}

Note that if $A \subseteq \MM$ is shift-invariant, then $A_0 := A \cap \MMo$ is rootshift invariant, and if $A_0 \subseteq \MMo$ is rootshift-invariant, then $A := GA_0$ is shift invariant. More is true:

\begin{prop}\label{transferprinciple}

Let $\mu$ be a point process with Palm measure $\mu_0$. 
\begin{enumerate}
    \item If $A \subseteq \MMo$ is rootshift invariant, then $\mu_0(A) = \mu(GA)$.
    \item If $A \subseteq \MM$ is shift invariant, then $\mu_0(A \cap \MMo) = \mu(A)$.
\end{enumerate}

That is, under the correspondence between rootshift invariant subsets of $\MMo$ and shift invariant subsets of $\MM$, the measures $\mu_0$ and $\mu$ coincide.

In particular, $G \acts (\MM, \mu)$ is ergodic \emph{if and only if} $(\MMo, \Rel, \mu_0)$ is ergodic.

\end{prop}

\begin{proof}
    We assume ergodicity and prove the statements about measures. The general case will follow.
    
    First, suppose $G \acts (\MM, \mu)$ is ergodic, and let $A \subseteq \MMo$ be rootshift invariant. Then for any $U \subseteq G$ of unit volume,
    \begin{align*}
        \mu_0(A) &= \frac{1}{\intensity \mu} \EE_\mu\left[ \#\{g \in U \mid g^{-1}\omega \in A \} \right] && \text{by definition} \\
        &= \frac{1}{\intensity \mu} \EE_\mu\left[ \abs{\omega \cap U} \1_{\omega \in GA} \right] && \text{by rootshift invariance of } A   \\
        &= \mu(GA) && \text{by ergodicity}.
    \end{align*}
    
    In particular, we see that $\mu_0(A)$ is zero or one, so the equivalence relation is ergodic.
    
    Now suppose $(\MMo, \Rel, \mu)$ is ergodic, and let $A \subseteq \MM$ be shift invariant.
    \begin{align*}
        \mu_0(A \cap \MMo) &= \frac{1}{\intensity \mu} \EE_\mu \left[ \#\{g \in U \mid g^{-1}\omega \in A \cap \MMo \} \right] && \text{by definition} \\
        &= \frac{1}{\intensity \mu} \EE_\mu\left[ \abs{\omega \cap U} \1_{\omega \in A} \right] && \text{by shift invariance of } A \\
        &= \mu(A) && \text{by ergodicity}.
    \end{align*}

For the general case, we appeal to the ergodic decomposition theorem (see \cite{MR1784210} for a proof):
\begin{thm}

Let $G$ be an lcsc group, and $G \acts (X, \mu)$ a pmp action on a standard Borel space. Then there exists a standard Borel space $Y$ equipped with a probability measure $\nu$ and a family $\{ p_y \mid y \in Y\}$ of probability measures $p_y$ on $X$ with the following properties:
\begin{enumerate}
    \item For every Borel $A \subset X$, the map $y \mapsto p_y(A)$ is Borel, and
    \[
        \mu(A) = \int_Y p_y(A) d\nu(y).
    \]
    \item For every $y \in Y$, $p_y$ is an invariant and ergodic measure for the action $G \acts (X, p_y)$,
    \item If $y, y' \in Y$ are distinct, then $p_y$ and $p_y'$ are mutually singular.
\end{enumerate}

\end{thm}

There is an almost identically stated version of the above theorem for pmp cbers as well. These two decompositions are essentially equivalent, in a way that we shall now discuss.

If $(Y, \nu)$ and $\{p_y \mid y \in Y\}$ is the ergodic decomposition for $G \acts (\MM, \mu)$, then the Palm measures $(p_y)_0$ of the $p_y$ form an ergodic decomposition for $(\MMo, \Rel, \mu_0)$. That is, for all $A \subseteq \MMo$ we have

\[
    \mu_0(A) = \int_Y (p_y)_0(A) d\nu(y).
\]
Applying the previous ergodic case to this yields the general formula.
\end{proof}

\begin{remark}

It is immediate that the ergodic decomposition for $G \acts (\MM, \mu)$ determines the ergodic decomposition for $(\MMo, \Rel, \mu_0)$. 

In the other direction, let $\{p'_y \mid y \in Y'\}$ denote the ergodic decomposition of $(\MMo, \Rel, \mu_0)$, so that
\[
    \mu_0(A) = \int_{Y'} p'_y(A) d\nu'(y).
\]
It turns out that all of the ergodic components $p'_y$ are not just probability measures on $\MMo$, but are themselves the Palm measures of point processes. This can be proven by using a characterisation of Mecke, see Theorem 13.2.VIII of \cite{daley2007introduction} (one applies the formula listed as item (iii) to $\support(p'_y)$).

One can then use the Voronoi inversion technique as referenced in Remark \ref{VIF} to construct the ergodic decomposition of $\mu$ out of the ergodic decomposition of $\mu_0$ (with an additional $\texttt{Unif}[0,1]$ random variable).

\end{remark}

We now prove Theorem \ref{correspondencetheorem}, building on Section \ref{borelcorrespondences}. The task here is to verify that under the correspondence, objects which are equal almost everywhere with respect to the point process are equal almost everywhere with respect to the Palm measure, and vice versa.

\begin{lem}\label{extensionlemma}
Let $\mu$ be a point process on $G$ with Palm measure $\mu_0$, and $X$ a Borel $G$-space.

Let $\Phi, \Phi' : \MM \to X$ be an equivariant Borel map. Then
\[
    \Phi = \Phi' \;\; \mu \text{ almost everywhere \emph{if and only if} } \restr{\Phi}{\MMo} = \restr{\Phi'}{\MMo}\;\; \mu_0 \text{ almost everywhere}.
\]
\end{lem}

\begin{proof}
Observe that by equivariance the sets
\[
    \{ \omega \in \MM \mid \Phi(\omega) = \Phi'(\omega)\} \text{ and } \{ \omega \in \MMo \mid \Phi(\omega) = \Phi'(\omega) \}
\]
are shift invariant and rootshift invariant respectively. So by Proposition \ref{transferprinciple} one is $\mu$-sure if and only if the other is $\mu_0$-sure, as desired.
\end{proof}

\begin{proof}[Proof of Theorem \ref{correspondencetheorem}]

The method is essentially the same for thinnings and for markings, so we will just prove the thinning statement. To that end, let $\theta : (\MM, \mu) \to \MM$ be a thinning. Note that by our assumption that $\theta$ is equivariant, we have
    \[
        \{ \omega \in \MM \mid \theta(\omega) \subseteq \omega \} \text{ has } \mu \text{ measure one}.
    \]
    This is a shift invariant set, so by Proposition \ref{transferprinciple} we have
    \[
        \{ \omega \in \MMo \mid \theta(\omega) \subseteq \omega \} \text{ has } \mu_0 \text{ measure one}.
    \]
We are now able to define $A = \{\omega \in \MMo \mid 0 \in \theta(\omega) \}$, and this will be our desired subset of $(\MMo, \mu_0)$. 

It follows from equivariance that the thinning $\theta^A$ associated to $A$ satisfies
\[
    \restr{\theta^A}{\MMo} = \restr{\theta}{\MMo} \;\; \mu_0 \text{ almost everywhere,}
\]
so by Lemma \ref{extensionlemma} we have $\theta^A = \theta$ ($\mu$ almost everywhere).

It remains to verify that if $A = B$ $\mu_0$ almost everywhere (that is, that $\mu_0(A \triangle B) = 0$, then $\theta^A = \theta^B$ ($\mu$ almost everywhere).

Recall\footnote{This is a general fact about nonsingular cbers, and it follows from the fact that they can all be generated by actions of \emph{countable} groups.} that the \emph{saturation} of $A \triangle B$
\[
    [A \triangle B] = \{ g^{-1}\omega \in \MMo \mid \omega \in A \triangle B \text{ and } g \in \omega \}
\]
is $\mu_0$ null if $A \triangle B$ is $\mu_0$ null.

Observe that for $\omega \not\in [A \triangle B]$ we have $\theta^A(\omega) = \theta^B(\omega)$, and hence $\restr{\theta^A}{\MMo} = \restr{\theta^B}{\MMo}$ $\mu_0$ almost everywhere, and we are finish by again applying Lemma \ref{extensionlemma}.

If $\mathscr{G}$ is a factor graph of $\mu$, then in the same fashion we see that it has a well-defined restriction to $(\MMo, \mu_0)$. We then define
\[
    \mathscr{A} = \{ (\omega, g) \in \MMo \times G \mid (0, g) \in \mathscr{G}(\omega) \}.
\]

We must verify that if $\mathscr{A}, \mathscr{B} \subseteq \Marrow$ are subsets with $\muarrow(A \triangle B) = 0$, then their associated factor graphs $\mathscr{G}^{\mathscr{A}}$ and $\mathscr{G}^{\mathscr{B}}$ are equal $\mu$ almost everywhere. This assumption states
\[
    \int_{\MMo} \# \{g \in \omega \mid (\omega, g) \in \mathscr{A} \triangle \mathscr{B} \} d\mu_0(\omega) = 0
\]
and hence the integrand is zero $\mu_0$ almost everywhere. By again considering the saturation of sets, we see that
\[
    \mu_0( \{\omega \in \MMo \mid \text{ for all } g \in \omega, g^{-1}\omega \in \mathscr{A} \triangle \mathscr{B} \}) = 0,
\]
from which the argument finishes as in the case of thinnings.
\end{proof}

\begin{remark}\label{cbercorrespondence}

If $\Pi$ is a free point process, then its Palm measure $\mu_0$ concentrates on the set $\Maper_0$ of aperiodic configurations (see Remark \ref{aperiodic}). Thus we only need to consider the rerooting equivalence relation. By using the canonical parametrisation $p_\omega : \omega \to [\omega]_\Rel$ for $\omega \in \Maper_0$, we can transfer any graph with vertex set $\omega$ to be one with vertex set $[\omega]_\Rel$ in a well-defined way. So we see that \emph{for free point processes}, factor graphs are the same thing as Borel graphings of the Palm equivalence relation. In the same way, given a group $\Gamma$ finitely generated by $S \subset \Gamma$ and a free pmp action $\Gamma \acts (\MMo, \mu_0)$ generating the Palm equivalence relation $\Rel$, we get a connected factor graph of $\Pi$ which is directed and edge-labelled by $S$, isomorphic to $\Cay(\Gamma, S)$. This correspondence goes both ways.

\end{remark}

\section{Cayley factor graphs}\label{cayleyfactorgraphs}
\subsection{Characterising amenability and constructing amenable Cayley graphs as factors}\label{amenability}

In this section, we will characterise amenability of a group in terms of the free point processes on it. Whilst not especially novel, this will clarify certain results in the literature. As an application of this we are able to construct essentially arbitrary Cayley factor graphs of amenable discrete groups on point processes in amenable groups.

Holroyd and Peres introduced the following concept in \cite{holroyd2003}:

\begin{defn}\label{clumping}
    
    Let $\Pi$ be a point process with law $\mu$. A sequence of factor graphs $\sim_n^\bullet : (\MM, \mu) \to \graph(G)$ is a \emph{one-ended clumping} if it satisfies the following for $\mu$ almost every $\omega \in \MM$:
    \begin{itemize}
        \item (Ascending) $\sim_1^\omega \subseteq \sim_2^\omega \subseteq \cdots$
        \item (Partitions) the connected components of each $\sim_n^\omega$ consist of \emph{finite} complete graphs, and
        \item (One-endedness) for all $x, y$ in $\omega$ there exists $N = N(x, y, \omega)$ such that $x$ is connected to $y$ in $\sim_N^\omega$.
    \end{itemize}
    
\end{defn}

We view $\sim_n^\omega$ as an \emph{equivalence relation} on $\omega$ consisting of finite classes. If $x, y \in \omega$ then we will write $x \sim_n^\omega y$ to denote that $x$ and $y$ are connected in $\sim_n^\omega$.

We explain one way of interpreting the following definition using the concept of \emph{Voronoi tessellations}. Recall:

\begin{defn}\label{voronoidefn}

    Let $\omega \in \MM$ be a configuration, and $g \in \omega$ one of its points. The associated \emph{Voronoi cell} is
    \[
        V_\omega(g) = \{ x \in G \mid d(x, g) \leq d(x, h) \text{ for all } h \in \omega \}.
    \]
    The associated \emph{Voronoi tessellation} is the ensemble of closed sets $\{V_\omega(g)\}_{g \in \omega}$.

\end{defn}

Left-invariance of the metric $d$ implies that the Voronoi cells are equivariant in the sense that for all $\gamma \in G$, we have $V_{\gamma \omega}(\gamma g) = \gamma V_\omega(g)$.

Note that discreteness of the configuration implies that the Voronoi tessellation forms a locally finite \emph{cover} of the ambient space by closed sets. We would like to think of these sets as forming a \emph{partition} of the ambient space, but this isn't necessarily true even in the measured sense: the boundaries of the Voronoi cells can have positive volume. For example, let $\Gamma$ be a discrete group and consider $\Gamma \times \{0\} \subset \Gamma \times \RR$. 

Lie groups and Riemannian symmetric spaces essentially avoid this deficiency, as hyperplanes\footnote{sets of the form $\{x \in X \mid d(x, g) = d(x, h) \}$ for a fixed distinct pair $g, h \in X$} have zero volume.

So depending on the examples one is interested in one can assume that the Voronoi cells are essentially disjoint (that is, that their intersection is Haar null). If this property is necessary then one can make a small modification to ensure it: we introduce a \emph{tie breaking} function that allows points belonging to multiple Voronoi cells to decide which one they shall belong to. Take any\footnote{Recall that standard Borel spaces are isomorphic if they have the same cardinality} Borel isomorphism $T : G \to \RR$. Let us define
    \[
        V_\omega^T(g) = \{ x \in G \mid d(x, g) \leq d(x, h) \text{ for all } h \in \omega \text{, and for all } h \in \omega \setminus \{g\}, T(x^{-1}g) < T(x^{-1}h) \}.
    \]

Note that these tie-broken Voronoi cells form a \emph{measurable} partition of $G$. That is, we have traded the Voronoi cells being closed for them being genuinely disjoint. The equivariance property $V^T_{\gamma \omega}(\gamma g) = \gamma V^T_\omega(g)$ still holds as well. For simplicity we will omit the tie-breaking function from the notation.

Then if $\Pi$ is a point process, then we view the ensemble of Voronoi cells $\{ V_\Pi(g) \}_{g \in \Pi}$ as a random measurable partition of $G$. If $\sim_n^\bullet$ is a clumping of $\Pi$, then it gives us a way to \emph{coarsen} the Voronoi partitioning as follows: for each $n$, define
\[
    \mathcal{P}_n = \left\{ \bigcup_{h \sim_n^\Pi g} V_\Pi(h) \right\}_{g \in \Pi}.
\]
Note that $\mathcal{P}_n$ is a refinement of $\mathcal{P}_{n+1}$. See Figure \ref{clumpingfigure}.

\begin{figure}[h]\label{clumpingfigure}
\includegraphics[scale=0.5]{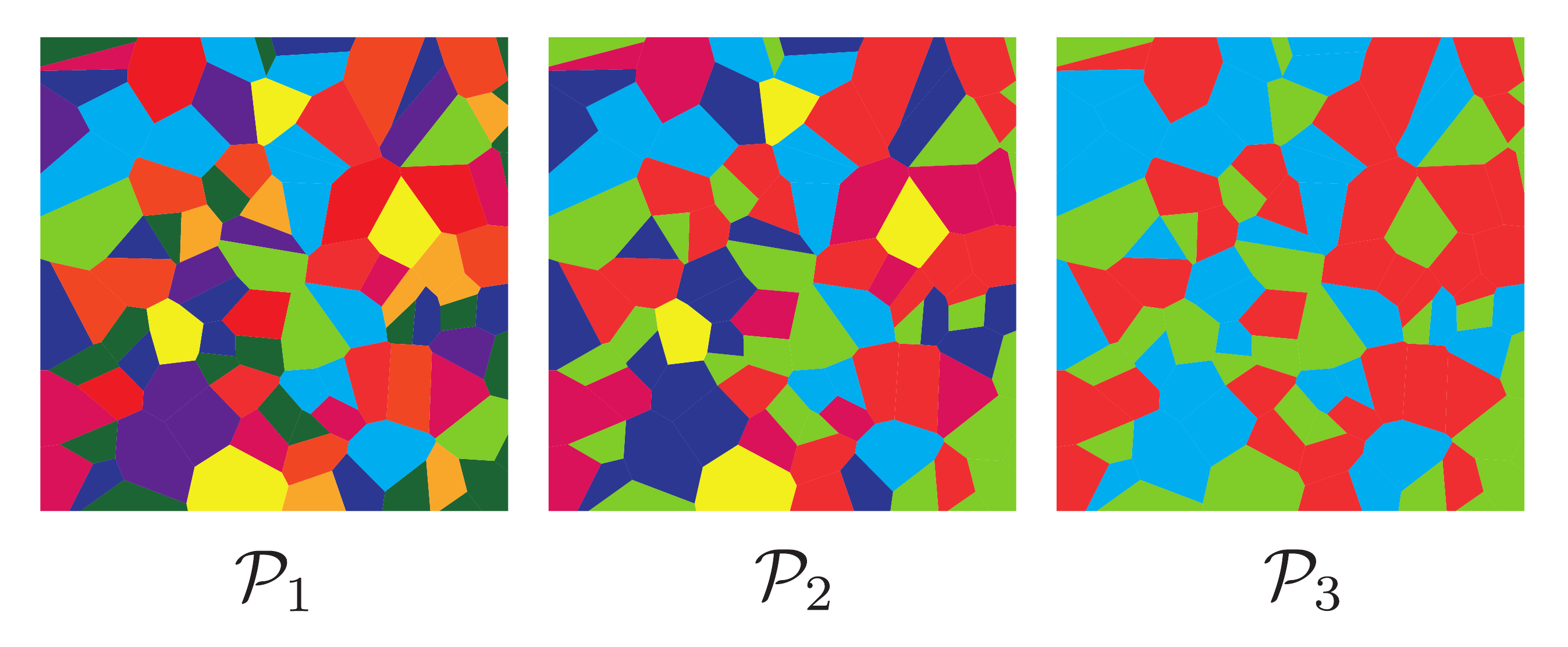}
\centering
\caption{This is how you should visualise the partitions associated to a one-ended clumping: like a sequence of worse and worse mosaics. Here two Voronoi cells receive the same colour if they are in the same equivalence class.}
\end{figure}

Holroyd and Peres were interested in (among other things) constructing particular kinds of connected factor graphs on the Poisson point process on $\RR^n$. Namely, they were interested in constructing one-ended factor trees and directed $\ZZ$s. They proved:

\begin{thm}[Holroyd-Peres\cite{holroyd2003}]
Let $\Pi$ denote a free and ergodic point process in $\RR^n$. Then the following are equivalent:
\begin{itemize}
    \item $\Pi$ admits a locally finite factor graph which is a connected and one-ended tree,
    \item $\Pi$ admits a factor graph which is isomorphic to the directed line $\ZZ$, and
    \item $\Pi$ admits a one-ended clumping.
\end{itemize}
Moreover, the Poisson point process admits a one-ended clumping.

\end{thm}

This was later extended by \'{A}d\'{a}m Tim\'{a}r, who also answered a question of Steve Evans about possible factor graph structures on point processes:

\begin{thm}[Tim\'{a}r \cite{timar2004}]
Let $\Pi$ denote a free and ergodic point process in $\RR^n$. Then $\Pi$ admits a one-ended clumping. Moreover, $\Pi$ admits a connected factor graph isomorphic to $\ZZ^d$, for any $d \in \NN$.
\end{thm}

It is clear from these works that the amenability of the underlying space $\RR^n$ is important, but the connection was not fully elucidated. We will prove

\begin{thm}\label{clumpingsexist}
    
    If $G$ is amenable, then all of its free point processes admit one-ended clumpings. 
    Conversely, if $G$ has a free point process that admits a one-ended clumping, then $G$ is amenable. The same is true for marked point processes.
    
\end{thm}

Recall the following:

\begin{defn}

A pmp cber $(X, \Rel, \mu)$ is \emph{$\mu$-hyperfinite} if there exists an increasing sequence $\Rel_1 \subseteq \Rel_2 \subseteq \cdots$ of subequivalence relations of $\Rel$ such that for $\mu$ almost every $x \in X$,
\begin{itemize}
    \item for all $n \in N$, $[x]_{\Rel_n}$ is finite, and 
    \item $[x]_\Rel = \bigcup_n [x]_{\Rel_n}$.
\end{itemize}

Denote by $\widetilde{\mu}$ the lifted measure of $\mu$ to $X \times X$.

A pmp cber $(X, \Rel, \mu)$ is \emph{$\mu$-amenable} if there exists for each $x \in X$ a normalised positive functional $p_x \in (\ell^\infty ([x]_\Rel))^*$ (a \emph{local mean}) such that $p_x = p_y$ for $\widetilde{\mu}$ almost every $(x, y) \in X \times X$, and such that the function $x \mapsto p_x\left( \restr{\phi}{[x]_\Rel}\right)$ is measurable for all $\phi \in L^\infty(X, \mu)$.
\end{defn}

In the measured category, these concepts are equivalent (see Chapter II Section 10 of \cite{kechris2004topics}):

\begin{thm}[Connes-Feldman-Weiss \cite{connes1981amenable}]

A pmp cber $(X, \Rel, \mu)$ is $\mu$-hyperfinite if and only if it is $\mu$-amenable.

\end{thm}

Under the correspondences we've described, a free point process $\Pi$ admits a one-ended clumping \emph{if and only if} its Palm equivalence relation $(\MMo, \Rel, \mu_0)$ is $\mu_0$-hyperfinite. This observation is what led to the present paper.

\begin{proof}[Proof of Theorem \ref{clumpingsexist}]
    The proof will be the same for marked and unmarked processes, so we work with unmarked ones for notational convenience.
    
    We first describe how a one-ended clumping can be used to construct an invariant mean on $G$ using a fairly standard technique, see Theorem 5.1 of \cite{benjamini1999group}.
    
    Let $\mu$ be a free point process, and fix a clumping $(\sim_n^\bullet)$ of it. If $f : G \to \RR$ is an essentially bounded function, define
    
    \[
        m_n(f) = \frac{1}{\intensity \mu} \EE_\mu \left[ \sum_{y \sim_n^\omega X(\omega)} \frac{f(y)}{\#\{ y \sim_n^\omega X(\omega)\} } \right],
    \]
    that is, we average the values of $f$ over the points in $X(\omega)$'s $n$th equivalence class.
    
    By invariance of the point process, one can see that
     \[
        m_n(g \cdot f) = \frac{1}{\intensity \mu} \EE_\mu \left[ \sum_{y \sim_n^\omega X^g(\omega)} \frac{f(y)}{\#\{ y \sim_n^\omega X^g(\omega)\} } \right].
    \]
    
    One-endedness of the clumping implies that for $n$ sufficiently large, $\{ y \sim_n^\omega X(\omega)\} = \{ y \sim_n^\omega X^g(\omega)\}$. So any ultralimit of the $m_n$ defines a left-invariant mean on $G$.

    For the other implication, fix a left-invariant mean $m \in (L^\infty(G))^*$. Given a bounded and positive function $f : [\omega]_\Rel \to \RR$, we extend it to a function $F : G \to \RR$ by making $F$ constant on the Voronoi cells, and averaging the values for those $g \in G$ that belong to multiple Voronoi cells\footnote{Note that each point only belongs to finitely many Voronoi cells, by local finiteness of the configuration}. Define
    \[
        F(g) = \sum_{\{x \in \omega \mid g \in V_\omega(x)\}} \frac{f(x^{-1}\omega)}{\#\{x \in \omega \mid g \in V_\omega(x)\}}.
    \]
    Now for each $\omega \in \MMo$ we define a mean on $[\omega]_\Rel$ by $p_\omega(f) = m(F)$. Then $p_\omega$ only depends on the equivalence class of $\omega$ by left-invariance of $m$, and satisfies the measurability requirement.
\end{proof}

\begin{remark}

A version of this theorem was independently proved by Paquette in \cite{paquette2018distributional}. He looks specifically at invariant point processes on Riemannian symmetric spaces and (among other things) proves that the Delauney triangulation of any point process on such a space is a unimodular random network which is \emph{anchored amenable} if and only if the ambient space is amenable. 

\end{remark}

\begin{proof}[Proof of Theorem \ref{amenabletheorem}]

We have seen from Theorem \ref{clumpingsexist} that $G$ is amenable if and only if one (and then all) Palm equivalence relations of free point processes are hyperfinite almost everywhere.

Note that the equivalence classes in $(\MMo, \Rel, \mu_0)$ are infinite, as the ambient group is noncompact.

Let $\Gamma$ be an infinite amenable group, finitely generated by $S \subset \Gamma$. 

Since $(\MMo, \Rel, \mu_0)$ is $\mu_0$-hyperfinite we can apply the Ornstein-Weiss theorem (see \cite{kechris2004topics} Chapter 2 Section 6) to find an \emph{orbit equivalence} $\phi : (\MMo, \mu_0) \to ([0,1]^\Gamma, \text{Leb}^{\otimes \Gamma})$, that is, a measure space isomorphism satisfying $\phi([\omega]_\Rel) = \Gamma \phi(\omega)$ for $\mu_0$ almost every $\omega$. We simply use this isomorphism to transfer the graph, using the fact that $\omega$ is bijectively equivalent with its rerooting equivalence class $[\omega]_\Rel$: define
\[
    \mathscr{G}(\omega) = \{ (g, h) \in \omega \times \omega \mid \exists s \in S \text{ such that } \phi(g^{-1}\omega) = \phi(h^{-1}\omega) s \}.
\]
Then $\mathscr{G}$ is the desired factor graph.
\end{proof}

\subsubsection{Nonamenability, the Poisson point process, and spectral gap}

We now describe another point process theoretic characterisation of nonamenability.

\begin{defn}

Let $G \acts (X, \mu)$ be a measure preserving (mp) action. Its \emph{Koopman representation} is the unitary representation $\pi$ of $G$ on $L^2(X, \mu)$ defined by
\[
    (\pi(g) f)(x) := f(g^{-1} x).
\]
We simply write $L^2(X)$ if the measure $\mu$ is understood.

Let $L^2_0(X) = \{ f \in L^2(X) \mid \int_X f(x) d\mu(x) = 0 \}$ denote the $G$-invariant subspace of mean zero functions. Note that $L^2_0(X) = L^2(X)$ if the underlying measure $\mu$ has $\mu(X) = \infty$.

An \emph{almost invariant sequence} in $L^2_0(X)$ is a sequence of unit vectors $f_n$ such that
\[
    \lim_{n \to \infty} \norm{\pi(g) f_n - f_n} = 0 \text{ for all } g \in G.
\]

We say the the action $G \acts (X, \mu)$ \emph{has spectral gap} if it has no almost invariant sequences.

\end{defn}
For further details, see the survey paper of Bekka \cite{MR3888695}.

Recall that $G$ is amenable if and only if its regular representation contains an almost invariant sequence.

\begin{prop}\label{spectralgap}

A group $G$ is nonamenable if and only if the Poisson point process action $G \acts (\MM, \PPP)$ on it has spectral gap.

\end{prop}

If $G$ is discrete, then one should interpret the above statement as referring to the Bernoulli shift $G \acts (\{0,1\}^G, \texttt{Ber}(p)^{\otimes G})$. In this case, the proposition is proved by expressing $L^2_0(\{0,1\}^G)$ as a direct sum of copies of the regular representation $\ell^2(G)$ and subregular representations. See Section 2.3.1 of Kerr and Li's book \cite{kerr2016ergodic} for further details, and Lyons-Nazarov\cite{MR2825538} for a particularly cool application of this fact.

In the nondiscrete case we appeal to an alternative decomposition of $L^2(\MM, \PPP)$ proved by Last and Penrose in \cite{last2011poisson}.

If $\mathcal{H}$ is a Hilbert space over $\RR$, we denote its $n$th tensor power by $\mathcal{H}^{\otimes n}$, with the convention that $\mathcal{H}^0 = \RR$. We denote by $S^n(\mathcal{H})$ the subspace generated by the symmetric tensors.

The Koopman representation turns products of measure spaces into tensor products: that is, $L^2((X_1, \mu_1) \otimes (X_2, \mu_2)) = L^2(X_1, \mu_1) \otimes L^2(X_2, \mu_2)$. In the analogous identification for $L^2(G, \lambda)^{\otimes n}$, the symmetric tensors $S^n(L^2(G))$ are identified with the space of $L^2$ functions on $G^n$ which are invariant under permutation of their variables.

\begin{thm}[Last-Penrose\cite{last2011poisson}]

Let $\PPP$ denote the Poisson point process on $G$ of unit intensity. Then the Koopman representation decomposes (as a unitary representation) as
\[
    L^2(\MM, \PPP) = \bigoplus_{n \geq 0} S^n(L^2(G)).
\]
\end{thm}

\begin{remark}

It should be stressed that Last and Penrose work with Poisson point processes in full generality on more-or-less arbitrary measure spaces, not merely the special case of lcsc groups with Haar measure. In particular, one also gets a similar decomposition of the Koopman representation of the IID Poisson on $G$.

\end{remark}

\begin{proof}[Proof of Proposition \ref{spectralgap}]

Simply observe that $S^n(L^2(G))$ is a subrepresentation of $L^2(G)^{\otimes n}$ by definition, which is in turn a subrepresentation of $L^2(G)^{\oplus \NN}$. Thus
\[
    L^2_0(\MM, \PPP) \text{ is a subrepresentation of } L^2(G)^{\oplus \NN}.
\]
Now recall that a representation $\pi$ has almost invariant vectors if and only if $\pi^{\oplus \NN}$ does, finishing the proof.
\end{proof}

\begin{question}

There is a general method for associating pmp actions to unitary representations known as \emph{Gaussian actions}, see \cite{bekka2008kazhdan} and \cite{kerr2016ergodic} for further details. The Gaussian action associated to the regular representation of $G$ has the same Koopman representation as the Poisson on $G$. Is there an invariant that distinguishes these actions?

\end{question}

Note that in the discrete case, the Gaussian action associated to $\Gamma \acts \ell^2 \Gamma$ is the IID Bernoulli action $\Gamma \acts [0,1]^\Gamma$.

\subsection{Property (T) and the nonexistence of Cayley factor graphs}\label{tsection}

In this section we prove Theorem \ref{kazhdantheorem}. Let us start with an informal sketch of the argument:

Note that by the correspondences we've described, a directed factor graph of the form $\text{Cay}(\Gamma,S)$ of $\Pi$ is the same thing as a free p.m.p. action $\Gamma \acts (\MM, \mu_0)$ of the Palm equivalence relation (plus a choice of finite\footnote{In fact, the ambient group $G$ will have Property (T) if and only if the Palm equivalence relation does in an appropriate sense, in which case $\Gamma$ will also have Property (T) and thus be finitely generated automatically.} generating set $S$). This action induces a \emph{cocycle} $c : \Rel \to \Gamma$ of the rerooting equivalence relation in a standard way, namely $c(\omega, g^{-1}\omega)$ is the unique element of $\Gamma$ satisfying
\[
    c(\omega, g^{-1}\omega) \cdot \omega = g^{-1}\omega,
\]
where the left-hand side uses the $\Gamma$ action. Our aim is to \emph{lift} this cocycle up to the action groupoid $G \times (\mathbb{M}, \mu)$, apply Popa's cocycle superrigidity there, and ultimately find a contradiction. 

We will state the definitions required to formally understand the basic case of Popa's cocycle superrigidity that we use. For a better understanding of why it works see Alex Furman's survey \cite{furman2009survey} and ergodic theoretic retelling in \cite{furman2007popa}, and also the book of Kerr and Li \cite{kerr2016ergodic}.

\begin{defn}[Malleability]
Let $G \acts (X, \mu)$ be a pmp action. Recall that the \emph{weak topology} on $\Aut(X, \mu)$ is the weakest topology that makes all functions $T \mapsto \mu(TA)$ continuous, where $T \in \Aut(X, \mu)$ and $A \subseteq X$ is Borel.

The \emph{flip} element of $\Aut(X \times X, \mu \otimes \mu)$ is $\texttt{FLIP}(x, y) = (y, x)$. 

Note that $G$ acts on $\Aut(X \times X, \mu \otimes \mu)$ diagonally via $(g \cdot T)(x, y) := T(gx, gy)$, and $\texttt{FLIP}$ commutes with this action. 

The action $G \acts (X, \mu)$ is \emph{malleable} if there exists a continuous path $\gamma: [0, 1] \to \Aut(X \times X, \mu \otimes \mu)$ from $\id$ to $\texttt{FLIP}$ such that $\gamma_t$ commutes with the diagonal action for every $t \in [0,1]$.

\end{defn}

The following fact seems to have gone unobserved:

\begin{prop}

The IID Poisson point process is malleable.

\end{prop}

\begin{proof}

Observe that a sample from $[0,1]^\PPP \otimes [0,1]^\PPP$ (that is, sampling from two independent unit intensity IID Poissons and keeping track of which is which) is the same as sampling from an IID Poisson $\Pi$ of double the intensity with labels from $[0,1] \times \{\pm 1\}$.

Define for $0 \leq t \leq 1$ the map $t : [0,1] \times \{\pm\} \to [0,1] \times \{\pm\}$ by
\[
    \phi_t(x,i) = \begin{cases} (x, -i) & x \leq t \\ (x, i) & \text{ else.} \end{cases}
\]
Now define
\[
    \gamma_t(\Pi) = \{ (g, \phi_t(x,i)) \in G \times [0,1] \times \{\pm\} \mid (g, x, i) \in \Pi \}.
\]
Then $\gamma_t$ continuously deforms $\id$ to $\texttt{FLIP}$.
\end{proof}

Recall that a groupoid consists of a set of composable arrows $\mathcal{G}$ and a unit space $\mathcal{G}_0$. For our main case of interest this is $\Marrow$ and $\MMo$ respectively.

\begin{defn}

Let $\Gamma$ be a discrete group and $\mathcal{G}$ a groupoid. A \emph{$\Gamma$-valued cocycle} of the groupoid is a measurable function $c : \mathcal{G} \to \Gamma$ satisfying \emph{the cocycle identity}
\[
    c(g) \cdot c(h) = c(gh) \text{ for all } \omega \in \MMo \text{ and } g, gh \in \mathcal{G}.
\]

Two cocycles $c, c' : \mathcal{G} \to \Gamma$ are \emph{cohomologous} if there exists a measurable function $f : \mathcal{G}_0 \to \Gamma$ such that for all $g \in \mathcal{G}$
\[
    c'(g) = f(t(g)) c(g) f(s(g)).
\]

\end{defn}

\begin{remark}

Recall that in the categorical framework, a groupoid is a category where every arrow is invertible, and a group is the same thing but with only one object. In this language, a cocycle is a functor from a groupoid to a group, and two such cocycles are cohomologous exactly when there's a natural transformation between the two functors.

\end{remark}

\begin{example}

If $G \acts (X, \mu)$ is a pmp action, then the associated \emph{action groupoid} has unit space $(X, \mu)$ and arrow space $G \times (X, \mu)$. The source of such an arrow is $x$, and its target is $g^{-1}x$. The composition rule for arrows is
\[
    (g, x) \cdot (h, y) := (gh, x) \text{ if } y = g^{-1}x.
\]

Note that if $\rho : G \to \Gamma$ is a homomorphism, then it induces a cocycle $c_\rho(\omega, g) = \rho(g)$. We will abuse notation and denote this cocycle simply by $\rho$. 

In an identical way we see that $\rho$ can be viewed as a cocycle of $\Marrow$.

\end{example}

We will use the following very basic form of Popa's cocycle superrigidity theorem:

\begin{thm}[\cite{popa2008cocycle}]

Let $G \acts (X, \mu)$ be a malleable and weakly mixing pmp action of an lcsc group $G$ with Property (T). Then any cocycle $c : G \times X \to \Gamma$ of the action groupoid is cohomologous to a homomorphism $\rho : G \to \Gamma$.

\end{thm}

We will apply this theorem using the following induction process:

\begin{prop}\label{induction}

Let $\mu$ be an ergodic point process on a nondiscrete group $G$. Then there is a factor map (as measure preserving groupoids) from the action groupoid $G \times (\MM, \mu)$ to the Palm groupoid $(\MMo, \mu_0)$.

In particular, any cocycle of the Palm groupoid can be \emph{lifted} to a cocycle of the action groupoid.

\end{prop}

The induction procedure will require a bit more probability theory, which is a slight generalisation of work of Holroyd and Peres \cite{holroyd2005}.

\begin{defn}
Let $\Pi$ be an ergodic and invariant point process. A \emph{partial allocation} for $\Pi$ is a measurable and equivariantly defined function $\mathcal{A} : \mathbb{M} \times G \to G \cup \{\bullet\}$ with the property that if $\mathcal{A}(\omega, g) \in G$ then $\mathcal{A}(\omega, g) \in \omega$.

We think of a partial allocation as an equivariant and measurable assignment to each $x \in \omega$ a measurable subset $A_\omega(x) = \{g \in G \mid \mathcal{A}(\omega, g) = x \}$ of $G$. This is a piece of ``land'' apportioned to $x$. If $\mathcal{A}(\omega, g) = x$, then we think of the point $g \in G$ as being \emph{assigned} to $x$ when the current configuration is $\omega$. If $\mathcal{A}(\omega, g) = \bullet$ then we think of $g$ as being an infinitesimal piece of unclaimed land.

We are interested in partial allocations which are defined as factors of an invariant point process $\Pi$. We will consider two allocations $\mathcal{A}$ and $\mathcal{B}$ to be \emph{equivalent} if
\[
    \lambda(\{g \in G \mid \mathcal{A}(\omega, g) \neq \mathcal{B}(\omega, g) \}) = 0 \text{ for an almost sure set of } \omega \in \mathcal{M}.
\]

\end{defn}

If $\Pi_0$ is the Palm version of $\Pi$, then it is natural to consider $\EE\left[ \lambda(A_0(\Pi_0)\right]$, the expected volume of the land allocated to the identity. Intuitively this is at most $\intensity(\Pi)^{-1}$.

\begin{defn}
An allocation is \emph{proper} if $\lambda(A_\Pi(x)) \leq \intensity(\Pi)^{-1}$ for all $x \in \Pi$. It is \emph{balanced} if $\lambda(A_\Pi(x)) = \intensity(\Pi)^{-1}$ for all $x \in \Pi$.

\end{defn}

\begin{remark}
A balanced allocation is therefore an equivariantly defined factor partition (up to a Haar null set) of the group.
\end{remark}

Let us verify our intuition by showing that the cells of a factor partition of a process $\Pi$ have expected volume at most $\intensity(\Pi)^{-1}$ using the CLMM:
\begin{align*}
    \EE_0 [\lambda(\{ x \in A_{\Pi_0}(0)] &= \EE_0 [\lambda(\{ x^{-1} \in A_{\Pi_0}(0) \})] && \text{By unimodularity} \\
    &= \EE_0 \left[ \int_G \1[x^{-1} \in A_{\Pi_0}(0)] d\lambda(g) \right] \\
    &= \frac{1}{\intensity \Pi} \EE \left[ \sum_{x \in \omega} \1[x^{-1} \in A_{g^{-1}\Pi}(0)]\right] && \text{By the CLMM}\\
    &= \frac{1}{\intensity \Pi} \EE \left[ \sum_{x \in \omega} \1[0 \in A_\Pi(g)]\right] && \text{By equivariance}\\
    &= \frac{1}{\intensity \Pi},
\end{align*}
as we note that every term in the sum is zero except for one.

This shows that the expected volume of the cell of the identity in the Palm process is $\intensity(\Pi)^{-1}$, and as all the cells have the same volume, it must be exactly this value.

Aside from their intrinsic interest -- wouldn't it be swell to share everything equally? -- balanced allocations have other applications. 

\begin{defn}
    Let $\Pi$ be an invariant point process. An \emph{extra head scheme} for $\Pi$ is a measurable function $\mathcal{E} : \MM \to G$ such that $\mathcal{E}_\Pi^{-1} \Pi$ is a Palm version of $\Pi$. Note that $\mathcal{E}_\Pi \in \Pi$. 
    
    Equivalently, let $E : \MM \to \MMo$ be the map $E(\omega) = \mathcal{E}_\omega^{-1} \omega$. Then if $\mu$ is the distribution of $\Pi$ and $\mu_0$ the Palm measure of $\Pi$, we ask that $E_* \mu = \mu_0$.
\end{defn}

Our interest in extra head schemes is that they are a way of factoring the point process onto its own Palm measure \emph{whilst respecting orbit structure}, since $E(\omega) \in G\omega$. 

Note that if we simply define $\mathcal{E}_\Pi$ to be the point $g \in \Pi$ whose Voronoi cell contains the origin, then it will \emph{not} be an extra head scheme in general. The essential issue here is that the Voronoi cells have different volumes, and thus some form of size-biasing is required. This is illustrated in the following lemma, which is proved in the same fashion as the previous computation.

\begin{lem}
If $\mathcal{A}$ is a balanced allocation, then the function $X = X(\Pi)$ given by
\[
    X = \text{ the unique } g \in \Pi \text{ such that } 0 \in A_g
\]
is an extra head scheme.
\end{lem}

\begin{thm}
Let $\Pi$ be an ergodic and invariant point process on a nondiscrete group $G$ of finite intensity. Then a balanced allocation for $\Pi$ exists, and hence also an extra head scheme.
\end{thm}

\begin{proof}

We use a measurable Zorn's lemma style of argument. 

Let $\mathscr{A}$ denote the set of proper allocations, that is, those allocations $\mathcal{A}$ with $\lambda(A_x) \leq \intensity(\Pi)^{-1}$ for all $x \in 
\Pi$. We order this space in the following way: declare $\mathcal{A} \preceq \mathcal{B}$ if $A_x \subseteq B_x$ for all $x \in \Pi$.

We refer to the quantity
\[
    c(\mathcal{A}) = \EE\left[\lambda\left(U \cap \bigcup_{x \in \Pi} A_x\right)\right]
\]
as the \emph{coverage} of an allocation. 

We claim that a maximal allocation $\mathcal{A}_\infty$ exists, and that its coverage is one, so it is a balanced allocation.

To see that maximal allocations exist, let us define
\[
    C(\mathcal{A}) = \sup\{ c(\mathcal{B}) \mid \mathcal{A} \preceq \mathcal{B} \},
\]
and inductively define an allocation in the following way: let $\mathcal{A}_1$ be an arbitrary proper allocation (even empty), and choose $\mathcal{A}_{n+1}$ so that
\[
    c(\mathcal{A}_{n+1}) \leq C(\mathcal{A}_n) + \frac{1}{n}.
\]
Now let $\mathcal{A}_\infty$ be the union of the allocations $\mathcal{A}_n$ in the obvious sense. It is straightforward to see that $\mathcal{A}_\infty$ is a maximal allocation.

We now show that any proper allocation $\mathcal{A}$ with coverage strictly less than one is contained in a proper allocation $\mathcal{B}$ with $c(\mathcal{A}) < c(\mathcal{B})$, and consequently any maximal allocation is balanced.

If $x \in \Pi$ has $\lambda(A_x) < \intensity(\Pi)^{-1}$, then we refer to $x$ as \emph{wanting}. If $\mathcal{A}$ is not balanced, then there must exists points $g \in \Pi$ such that
\[
    V_\Pi(g) \not\subseteq \bigcup_{x \in \Pi} A_x,
\]
where $V$ denotes the \emph{tie-broken} Voronoi cells. We refer to these as \emph{sharers}.

The idea is simply that each point which is wanting will choose a sharer, and be allocated as much land as it can take. If multiple wanters apply to the same sharer, then we'll simply pick one lucky wanter, as it's enough for the proof. The only trick is to do all this in a measurable and equivariant fashion.

Let us fix an isomorphism as measure spaces $I : (G, \lambda) \to ([0, \infty), \texttt{Leb})$. Then $I_x(g) = I(x^{-1}g)$ is also an isomorphism of $(G, \lambda)$ with $([0, \infty), \texttt{Leb})$, but has the virtue of being equivariantly defined as $x$ varies.

Each point which is wanting \emph{applies} to its nearest sharer, and then the sharer picks one of these by choosing the closest wanter (and if this is not unique, it uses a tie-breaking function in the usual way to select one). Suppose $s \in \Pi$ is a sharer that chooses the wanter $w \in \Pi$. Choose $t > 0$ so that
\[
0 < \lambda\left((I_s^{-1}([0, t)) \cap V_\Pi(s) \cap \left(\bigcup_{x \in \Pi} A_x\right)^c\right) \leq \intensity(\Pi)^{-1} - \lambda(A_w),
\]
for instance by enumerating the positive rationals and choosing the first $t \in \mathbb{Q}$ for which the above is true.
We now define a new allocation $\mathcal{B}$ by declaring $B_x = A_x$ for all points except the lucky wanters $w \in \Pi$, for which
\[
    B_w = A_w \cup I_s^{-1}([0, t)) \cap V_\Pi(s) \cap \left(\bigcup_{x \in \Pi} A_x\right)^c.
\] 
Then $\mathcal{A} \prec \mathcal{B}$ strictly, as desired.
\end{proof}

\begin{remark}\label{groupoidmap}
The extra head scheme gives us a \emph{factor} map from the unit space of the action groupoid to the Palm groupoid. We now extend this to a factor map of their arrow spaces in the following way: define
\begin{align*}
    &\overrightarrow{E} : G \times (\MM, \mu) \to (\Marrow, \muarrow) \\
    &\overrightarrow{E}(\omega, g) = (E(\omega), \mathcal{E}_\omega^{-1} g \mathcal{E}_{g^{-1}\omega})
\end{align*}

One can readily verify that this map preserves the source map, target map, and composition rule. 

The following diagram (pictured in $\RR^2$) explains what is going on:

\begin{figure}[h]
\includegraphics[scale=0.3]{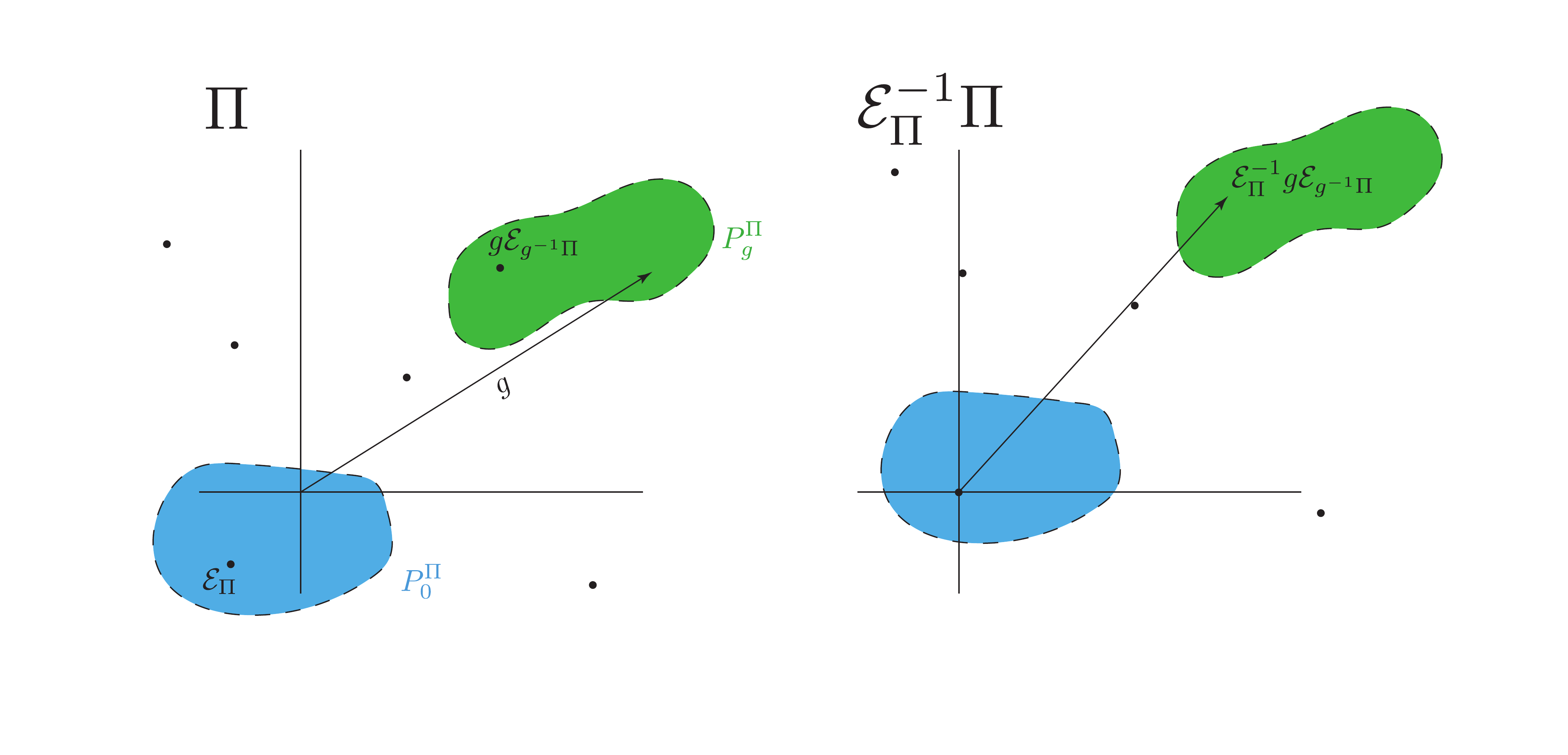}
\centering
\end{figure}

An element of the action groupoid $G \times \MM$ consists of a pair $(g, \omega)$. We view the configuration as a subset of $G$, and $g$ as an arrow pointing from the identity to $g$. 

This arrow lands in some cell of the equitable\footnote{In reality the partition will be much messier than the diagram.} partition, here shaded green. The map $\overrightarrow{E}$ simply treats this arrow as one going between the germs of the cells. The induction map is thus a kind of ``discretisation'' process.
\end{remark}

\begin{remark}

As a formal computation, the above seems to work for \emph{any} measurably defined partition (for instance, the Voronoi cells). The crucial feature of the extra head scheme is that it pushes forward the measure to the right measure.

\end{remark}

\begin{proof}[Proof of Theorem \ref{kazhdantheorem}]

If a point process admits no factor of IID connected Cayley graphs, then it certainly doesn't admit any deterministic ones, so we focus on the stronger statement. Let us write $\Pi$ for the Poisson point process and $\mu$ for its law. Then by the discussion at Remark \ref{cbercorrespondence} it suffices to show that there is no free pmp action $\Gamma \acts ([0,1]^\MMo, [0,1]^{\mu_0})$ which generates the rerooting equivalence relation $\Rel$ for \emph{any} countable group $\Gamma$.

So for the sake of contradiction we suppose that there is such an action. This defines a $\Gamma$-valued cocycle of the Palm groupoid in the following way\footnote{This is a standard construction known as the \emph{orbit equivalence cocycle}, although there's an extra inverse as an artifact of our conventions}:
\[
    c(\omega, g) = \gamma \text{, where } \gamma \in \Gamma \text{ is the \emph{unique} element such that } \gamma^{-1} \cdot \omega = g^{-1}\omega. 
\]
By the extra head scheme technique we can induce this to a cocycle $C : G \times ([0,1]^\MM, [0,1]^{\mu}) \to \Gamma$. Explicitly,
\[
    C(g, \omega) = \gamma, \text{ where } \gamma \in \Gamma \text{ is the unique element satisfying } \gamma^{-1} \cdot E(\omega) = E(g^{-1} \omega).
\]

Now we may apply Popa's cocycle superrigidity to find a homomorphism $\rho : G \to \Gamma$ and a measurable function $f : ([0,1]^\MM, [0,1]^{\mu}) \to \Gamma$ such that
\[
    C(g, \omega) = f(g\omega)\rho(g)f(\omega)^{-1}.
\]
By assumption $\ker(\rho)$ is noncompact. Note that for $g \in \ker \rho$, we have $C(g, \omega) = f(g \omega) f(\omega)^{-1}$. By definition of the cocycle then
\[
    f(g\omega)^{-1} E(g^{-1}\omega) = f(\omega)^{-1} E(\omega).
\]
That is, the function $\omega \mapsto f(\omega)^{-1} E(\omega)$ is $N$-invariant. The IID Poisson point process is a mixing action for $G$, and hence also for $N$ by noncompactness. Therefore this $N$-invariant function must be \emph{constant} by ergodicity. Note that $f(\omega)^{-1} E(\omega) \in [\omega]_\Rel$ for every $\omega$. Thus if this function is a constant $\Omega \in [0,1]^\MMo$, we would have $\PP[\Pi \in G \cdot \Omega] = 1$, but
\begin{align*}
    \PP[\Pi \in G \cdot \Omega] &= \PP[ \Pi_0 \in [\Omega]_\Rel] && \text{By Proposition \ref{transferprinciple}} \\
    &\leq \sum_{g \in \Omega} \PP[\Pi_0 = g^{-1}\Omega ] && \text{By definition of } \Rel \\
    &= 0, && 
\end{align*}
where the last line follows from the fact that the Palm measure of the Poisson has no atoms (see Theorem \ref{palmofpoisson}).
\end{proof}

The above proof can be pushed a little further:

\begin{thm}\label{theoremextension}
    Let $G$ be a locally compact and second countable nondiscrete group with Kazhdan's Property (T), and $\Pi$ be the Poisson point process on $G$. Assume further that $G$ has no compact normal subgroups.

    Then no thickening or thinning of $\Pi$ of finite intensity admits connected Cayley factor graphs, or even factor of IID connected Cayley factor graphs.
    
    Additionally, the Palm equivalence relation of $\Pi$ has the property that no induction or amplification of it can be freely generated by an action of a discrete group.
\end{thm}

We review the notions of induction and amplification for pmp cbers.

Let $(X, \mu, \Rel)$ be an ergodic pmp cber. If $A \subseteq X$ is measurable, then we write $\mu_A$ for the conditional measure
\[
    \mu_A(B) = \frac{\mu(A \cap B)}{\mu(A)},
\]
and $\Rel_A = \Rel \cap (A \times A)$ for the restricted equivalence relation. The resulting pmp cber $(A, \mu_A, \Rel_A)$ depends only on $t = \mu(A)$ by ergodicity, and the resulting pmp cber is denoted $\Rel_t$ and is called the \emph{induced} equivalence relation. This definition can be extended in a well-defined way for $t > 1$ and is referred to as \emph{amplification}: for integral $t$ we define $\Rel_t$ on $X \times [t]$ by
\[
    ((x, i), (y, j)) \in \Rel_t \text{ if } (x, y) \in \Rel,
\]
and use the product of $\mu$ with (normalised) counting measure on $[t]$. By combining induction and amplification, one defines $\Rel_t$ for arbitrary $t > 0$ in a well-defined way.

When the relation in question is the Palm equivalence relation of some point process $\Pi$ with distribution $\mu$, we are able to visualise inductions and amplifications concretely.

For an induction determined by a subset $A \subseteq \MMo$, we look at the associated $2$-colouring $\mathscr{C}_A$ (where $A$ points are coloured red and not-$A$ points are coloured blue). Then the equivalence relation consists of rooted configurations chosen according to $\mathscr{C}_A(\Pi_0)$, conditioned on the root being red, and one is allowed to shift the root \emph{only} to other red points.

For the amplification determined by $t \in \NN$, we look at the point process $\Pi \times [t] \subset G \times [t]$, and consider it as a $G \times \Sym(t)$ action. Here the appropriate groupoid to consider consists of configurations $\omega \subset G \times [t]$ rooted at the identity $0 \in G$ \emph{and} a particular level $l \in [t]$. That is, we use
\[
    \MMo^{[t]} = \{ (\omega, l) \in \MM(G \times [t]) \times [t] \mid (0, l) \in \omega \}
\]
as the unit space for the rerooting groupoid. If $A \subseteq \MMo$, then we also write 
\[
    A^{[t]} = \{ (\omega, l) \in \MMo^{[t]} \mid \pi(\omega) \in A \},
\]
where $\pi(\omega) \in \MM(G)$ is simply $\omega$ with the labels removed. 

In order to prove Theorem \ref{theoremextension}, we follow the same strategy as the above proof but simply arrive at a different contradiction. To that end, let us introduce the following definition:

\begin{defn}

Let $\Pi$ be an invariant point process. We say that $\Pi$ \emph{concentrates on a single orbit} if there exists a rooted configuration $\Omega \in \MMo$ such that $\PP[\Pi \in G.\Omega] = 1$.

\end{defn}

\begin{lem}

If $\Pi$ is an ergodic point process of finite intensity and it concentrates on a single orbit, then $\Pi$ is a lattice shift or a thickening of a lattice shift.

\end{lem}

\begin{proof}

By assumption and shift invariance of the event $\{\Pi \in G.\Omega\}$, we have that $\PP[\Pi_0 \in [\Omega]_\Rel] = 1$. By mass transport,
\[
    \PP[\Pi_0 = \Omega] = \PP[\Pi_0 = g^{-1}\Omega] \text{ for all } g \in \Omega,
\]
and so there exist finitely many $\Omega_1 = \Omega, \Omega_2, \ldots, \Omega_k \in \MMo$ such that for all $g \in \Omega$ there exists $i$ with $g^{-1}\Omega = \Omega_i$. 

We claim that $\Pi$ is a thickening of the \emph{lattice} $\Gamma = \stab(\Omega)$. 

First, $\stab(\Gamma) \subseteq \Omega$, so it is certainly discrete. Then
\[
    \theta(\Pi) = \{g \in \Pi \mid \stab(g^{-1}\Pi) = \Gamma \}
\]
is a $G$-invariant process supported on the cosets of $\Gamma$, hence $\Gamma$ has finite covolume, as desired.
\end{proof}

\begin{lem}\label{noconcentration}
No finite intensity thickening or thinning of the Poisson point process concentrates on a single orbit.
\end{lem}

\begin{proof}

If a thickening of the Poisson point process concentrates on a single orbit, then it would imply that there is a \emph{discrete} subgroup $\Gamma$ which contains the differences of all pairs of points from a sample of the Poisson point process. But this is impossible: take infinitely many disjoint unit balls $B_i$ in $G$. Then almost surely for every $k \in \NN$ there exists a ball $B_{i_k}$ of radius one such that $\abs{B_{i_k} \cap \Pi} \geq k$. In particular, the differences from these specific elements of $\Gamma$ will be nondiscrete in the unit ball of $G$.

If a thinning of the Poisson point process concentrates on a single orbit, then there would be an $R > 0$ such that
\[
    \PP\left[ \exists \text{ distinct } x, y \in B(0, R) \cap \Pi_0 \setminus \{0\} \text{ such that } x^{-1}y \in \Pi_0 \right] > 0,
\]
where $\Pi_0$ denotes the Palm version of the Poisson point process. But this is impossible.
\end{proof}

\begin{proof}[Proof of Theorem \ref{theoremextension}]

If $\Phi$ is a thickening or thinning of $\Pi$ with Palm equivalence relation freely generated by $\Gamma$, then there is a sequence of groupoid maps
\[
    G \times ([0,1]^\MM, \mu) \to G \times (\MM, \Phi_* \mu) \to (\Marrow, \overrightarrow{(\Phi_* \mu)_0}) \to \Gamma,
\]
where the first arrow is induced from $\Phi$ itself, the second arrow is from the extra head scheme for $\Phi(\Pi)$, and the final arrow is from the cocycle $c : (\Marrow, \overrightarrow{(\Phi_* \mu)_0}) \to \Gamma$. We let $C$ denote the composition of these three maps.

Again by Popa's cocycle superrigidity there exists a homomorphism $\rho : G \to \Gamma$ and a measurable function $f : ([0,1]^\MM, [0,1]^{\mu}) \to \Gamma$ such that 
\[
    C(g, \omega) = f(g\omega)\rho(g)f(\omega)^{-1}.
\]
By assumption $\ker(\rho)$ is noncompact. Note that for $g \in \ker \rho$, we have $C(g, \omega) = f(g \omega) f(\omega)^{-1}$. By definition of the cocycle then
\[
    f(g\omega)^{-1} E(\Phi(g^{-1}\omega)) = f(\omega)^{-1} E(\Phi(\omega)).
\]
That is, the function $\omega \mapsto f(\omega)^{-1} E(\Phi(\omega))$ is $N$-invariant. The IID Poisson point process is a mixing action for $G$, and hence also for $N$ by noncompactness. Therefore this $N$-invariant function must be \emph{constant} by ergodicity. Note that $f(\omega)^{-1} E(\Phi(\omega)) \in [\Phi(\omega)]_\Rel$ for every $\omega$. Hence $\Phi_*\mu$ concentrates on a single orbit, a contradiction by Lemma \ref{noconcentration}.

We denote by $\Rel_t$ the induced (when $0 < t < 1$) or amplified (when $t > 1$) Palm equivalence relation of the IID Poisson point process. We now show that these cannot be freely generated by any action of a finitely generated discrete group $\Gamma$. 

Let $A \subseteq [0,1]^\MMo$ denote a subset of size $t$. Then the existence of a generating action $\Gamma \acts (A, [0,1]^{\mu_0}_A)$ of $\Rel_t$ is the same as the existence of a factor of IID factor graph of the Poisson point process which lives on the set $\{g \in \Pi \mid g^{-1}\Pi \in A \}$, where it is a copy of the Cayley graph of $\Gamma$. This gives us a cocycle $(A, [0,1]^{\mu_0}_A) \to \Gamma$.

We now repeat the argument as earlier, except the induction map $[0,1]^\MM \to A$ from the extra head scheme uses the extra head scheme for the thinned process $\theta_A(\Pi)$. In essence, we construct a balanced allocation as before, but only the $A$-points of the process are allocated land. The induced cocycle $G \times ([0,1]^\MM, [0,1]^{\mu}) \to \Gamma$ untwists by Popa's cocycle superrigidity, and this gives a contradiction as before.

Finally, even the above equivalence relation amplified by some integer $t \in \NN$ \emph{still} cannot be freely generated by the action of a countable group $\Gamma$. We denote this equivalence relation by $A^{[t]}$ as before. Such an action would give us a cocycle $c : A^{[t]} \to \Gamma$, which we could then induce to a cocycle $C : G  \times ([0,1]^\MM, [0,1]^{\mu}) \to \Gamma$. Here the induction additionally uses the labels: we map $\omega \in [0,1]^\MM$ to $(g^{-1}\omega, i)$, where $g^{-1}\omega$ is in $A$, and $0$ is in the cell of $g$ with respect to $\omega$, and the label $\xi_g$ of $g$ in $\omega$ satisfies $i/t < \xi_g < (i+1)/t$. The rest of the argument follows as previously.
\end{proof}

\begin{remark}
Pmp cbers with the property that none of their amplifications or inductions can be freely generated by actions of discrete groups have been known since Furman \cite{furman1999orbit}. See Section 7 of \cite{popa2008cocycle} for further discussion, and the paper itself for examples with the property that their so-called \emph{fundamental group} is $\RR_+$. 
\end{remark}

\begin{question}
What is the fundamental group of the Palm equivalence relation of the Poisson point process on an lcsc group with Property (T)?
\end{question}

\begin{question}[Mikl\'{o}s Ab\'{e}rt]
Suppose $\mu$ is an ergodic point process on a group $G$ with Property (T) and no compact normal subgroups. If there exists a free action $\Gamma \acts (\MMo, \mu_0)$ generating the Palm equivalence relation, must $\Gamma$ be a lattice in $G$ and the point process the corresponding lattice shift?
\end{question}

\begin{remark}

There is a notion of Property (T) for pmp cbers (it works just as well for $r$-discrete pmp groupoids), see \cite{furman2009survey}. Thus one can ask about an analogue of Theorem \ref{clumpingsexist} -- does a group have Property (T) if and only if the Palm equivalence relation of all of its free point processes have Property (T)?

One can readily show that if $G$ has Property (T), then so too will the Palm equivalence relation of any free point process. In the discrete world, Zimmer showed that for $\Rel = \Rel(\Gamma \acts (X, \mu))$ the orbit equivalence relation of a free and weakly mixing pmp action, then if $\Rel$ has Property (T) then so too does $\Gamma$. Anantharaman-Delaroche removed the weak mixing requirement in \cite{anantharaman2005cohomology}.

\end{remark}

\appendix
\section{Point processes versus cross-sections}\label{crosssectionappendix}

We have taken the perspective that point processes are an \emph{intrinsically interesting} class of pmp actions of lcsc groups to study. They are also a fairly general class:

\begin{prop}\label{representationtheorem}

Every free and pmp action of a \emph{nondiscrete} lcsc group $G$ on a standard Borel measure space $(X, \mu)$ is abstractly isomorphic to a finite intensity point process.

\end{prop}

This is similar to the following fact: let $\Gamma \acts (X, \mu)$ be a pmp action of a discrete group $\Gamma$. The \emph{symbolic dynamics} of this action is the map
\begin{align*}
    &\Sigma : (X, \mu) \to X^\Gamma \\
    &\Sigma_x(\gamma) = \gamma^{-1}x.
\end{align*}
This is an injective and equivariant map, so we may identify the action $\Gamma \acts (X, \mu)$ with the invariant colouring action $\Gamma \acts (X^\Gamma, \Sigma_* \mu)$.

In this way, we see that all pmp actions of discrete groups are isomorphic to invariant colourings\footnote{If desired, one can fix a Borel isomorphism $X \cong [0,1]$ so that the colouring space is the same for all actions}.

A standard technique in the study of free pmp actions of lcsc groups is to analyse their associated \emph{cross-sections}. This will gives an analogue of symbolic dynamics for nondiscrete groups.

\begin{defn}
Let $G \acts (X, \mu)$ be a pmp action on a standard Borel measure space $(X, \mu)$.

A \emph{discrete cross-section} for the action is a Borel subset $Y \subset X$ such that for $\mu$-every $x \in X$ the set $\{g \in G \mid g^{-1}x \in Y \}$ is a discrete and non-empty subset of $G$. 
\end{defn}

\begin{example}
The set $\MMo \subset \MM$ is a discrete cross-section \emph{for all} non-empty point process actions $G \acts (\MM, \mu)$.
\end{example}

There is a sense in which this $\MMo$ is the \emph{only} cross-section.

Fix such a cross-section $Y \subset X$. We associate to this data two maps
\begin{align*}
    &\mathcal{V} : (X, \mu) \to \MM &&  \mathscr{V} : (X, \mu) \to Y^\MM  \\
    &\mathcal{V}_x = \{g \in G \mid g^{-1}x \in Y \} &&  \mathscr{V}_x = \{(g, g^{-1}x) \in G \times Y \mid g^{-1}x \in Y \}.
\end{align*}

These are equivariant maps, and the second one is always injective. In particular\footnote{Recall that an \emph{injective} map between standard Borel spaces is always a Borel isomorphism onto its image}, we see that every action which admits a cross-section also admits a point process factor, and is isomorphic to a \emph{marked} point process.

Note that $\mathcal{V}^{-1}(\MMo) = Y$. In this way we see that \emph{a discrete cross-section is the same thing as an unmarked point process factor}.

\begin{remark}[Terminological discussion]

If $\mathcal{P}(\omega)$ is some property of discrete subsets $\omega$ in $G$, then we can investigate discrete cross-sections of actions $G \acts (X, \mu)$ such that the associated subset $\mathcal{V}_x$ satisfies $\mathcal{P}$ for $\mu$ almost every $x \in X$. 

For instance, $\mathcal{P}(\omega)$ might be the property ``$\omega$ is uniformly discrete'' or ``$\omega$ is a net''. We will refer to a discrete cross-section such that $\mathcal{P}(\mathcal{V}_x)$ is satisfied for $\mu$ almost every $x \in X$ as a \emph{$\mathcal{P}$ cross-section}.

To the author's knowledge, a term like ``discrete cross-section'' does not appear in the literature. One finds instead discussion of \emph{lacunary cross-sections} (an admittedly more romantic name for what would be uniformly discrete cross-section), or \emph{cocompact cross-sections} (which we would refer to as net cross-sections).

Note that if $G \acts (\MM, \mu)$ is the Poisson point process action, then $\MMo$ is \emph{not} a lacunary cross-section. It is for this reason that we feel the terminology should be modified slightly.

\end{remark}

\begin{thm}[\cite{MR417388}, see also \cite{MR3335405}]\label{crosssectionsexist}

Every free and \emph{nonsingular}\footnote{Recall that an action is \emph{nonsingular} if it preserves null sets, that is, if $\mu(A) = 0$ then $\mu(gA) = 0$ for all $g \in G$} action of an lcsc group on a standard probability space admits a discrete cross-section. Moreover, the cross-section can be chosen to be uniformly separated and even a net.

\end{thm}

\begin{remark}

In fact, cross-sections of actions are known to exist in great generality, see \cite{kechris2019theory} for further examples.

Our keen interest in \emph{free} actions is because it allows us to identify the orbit $Gx$ of any point $x \in X$ with $G$ itself. One can run into issues in the absence of this.

For instance, let $\RR \times \RR$ act on $\{\bullet\} \times \RR/\ZZ$ diagonally, where $\{\bullet\}$ denotes a singleton with trivial action.

Then $\{ (\bullet, 0) \}$ is a lacunary cross-section for the action. If we try to construct a map $\mathcal{V}$ as before, then we would map $(\bullet, x) \in \{\bullet\} \times \RR/\ZZ$ to the subset of $\RR^2$
\[
    \mathcal{V}_{(\bullet, x)} = \RR \times \{ x + \ZZ \}.
\]
In this way one has constructed a \emph{random closed set} as a factor of the action, but it is not a point process. 
\end{remark}

The following theorem is described as folklore in \cite{MR3335405}:

\begin{thm}[Folklore theorem, see Proposition 4.3 of \cite{MR3335405}]

Let $G$ be a unimodular lcsc group, and $G \acts (X, \mu)$ a pmp action on a standard Borel space. Fix a lacunary cross-section $Y \subset X$ for the action. Then:
\begin{enumerate}
    \item The orbit equivalence relation of $G \acts X$ restricts to a cber $\Rel$ on $Y$,
    \item There exists an $\Rel$-invariant probability measure $\nu$ on $Y$,
    \item The action $G \acts (X, \mu)$ is ergodic if and only if the cber $(Y, \Rel, \nu)$ is ergodic,
    \item The group $G$ is noncompact if and only if the cber is aperiodic $\nu$ almost everywhere, and
    \item The group $G$ is amenable if and only if the cber $(Y, \Rel, \nu)$ is amenable.
\end{enumerate}
\end{thm}

The mathematical content of Theorem \ref{correspondencetheorem} can be viewed as a rediscovery of the above theorem with different proofs, together with interpretation of factor constructions as objects living on the Palm groupoid.

Marked point processes can be a useful contrivance, but aren't strictly necessary:

\begin{prop}\label{abstractlyisom}
Every free point process $\Pi$ on a nondiscrete group with marks from a standard Borel space $\Xi$ is abstractly isomorphic to an \emph{unmarked} point process.
\end{prop}

It should be easy to convince oneself that such a proposition will be true, although the details will necessarily be somewhat messy and ad hoc. We call the technique used \emph{local encoding}, which is illustrated in the following example:

\begin{figure}[h]\label{localencode}
\includegraphics[scale=0.5]{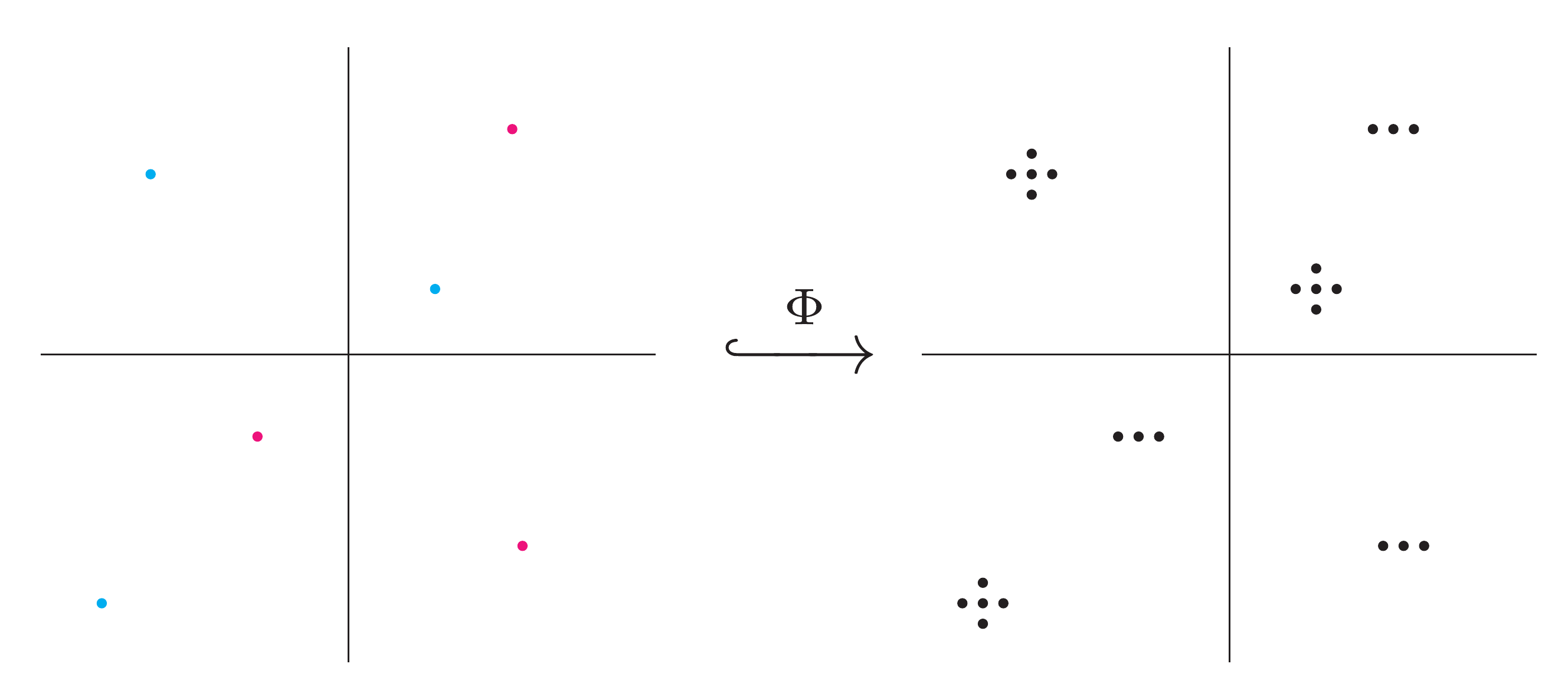}
\centering
\caption{Locally encoding labels of a point process.}
\end{figure}

This is a point process in $\RR^2$ labelled by the set $\{+, -\}$, which we have coloured as cyan and magenta respectively in the diagram. 

The map $\Phi : \{+, -\}^\MM \to \MM$ takes the input configuration, and adds a small decoration around each point. In this case we are literally encoding $+$ marks as a plus symbol centred at each point and similarly for $-$ marks. 

Barring some exceptional circumstances, you should be able to convince yourself that $\Phi$ is an injective map, and thus is an isomorphism onto its image for many input processes. 

The general case is achieved similarly. Note that the method here is necessarily ad hoc.

\begin{proof}[Proof of Proposition \ref{abstractlyisom}]

Suppose $\Pi$ is a free $\Xi$-marked point process with law $\mu$. We can assume that $\Pi$ is abstractly isomorphic to a uniformly separated process with a slightly different (but nevertheless standard Borel) mark space. One way to prove this is to simply appeal to Theorem \ref{crosssectionsexist} to choose a uniformly separated cross-section $Y \subset \Xi^\MM$, and then construct the equivariant injection
\[
    \mathscr{V} : (\Xi^\MM, \mu) \to Y^{\Xi^\MM}
\]
as discussed earlier (note also that we may identify $Y^{\Xi^\MM}$ with $(Y \times \Xi)^\MM$, so this is a standard Borel mark space). We then replace $\Pi$ by the isomorphic process $\mathscr{V}(\Pi)$.

Let $X$ denote the space:
\[
    X = \{ \omega \in \MMo(B(0, \delta/100)) \mid \omega \cap B(0, \delta/200) = \{0\}, \text{ and for all } x \in \omega \setminus \{0\}, \abs{B(x, \delta/200)} > 1 \}.
\]
This is a Borel subset of a standard Borel space, and hence standard Borel in its own right. One can readily see that it is uncountable, and hence there is a Borel isomorphism $I: \Xi \to X$. Define the following factor map:
\begin{align*}
    &\Psi : \Xi^\MM \to \MM \\
    &\Psi(\omega) = \bigcup_{x \in \omega} x I(\xi_x),
\end{align*}
where $\xi_x$ denotes the label of $x$ (that is, $(x, \xi_x) \in \omega$.

This is an injective map: we can recover the underlying set of any input configuration to $\Psi$ by identifying the points which are $\delta/200$-isolated. We can then uniquely recover their labels by applying the inverse of $I$ locally.
\end{proof}

\begin{proof}[Proof of Proposition \ref{representationtheorem}]
We have seen that a choice of cross-section gives us an isomorphic representation of the action as a marked point process. We then use Proposition \ref{abstractlyisom}.
\end{proof}

\begin{question}
Is there a more point process theoretic method to construct discrete cross-sections of free pmp actions?
\end{question}

\bibliographystyle{alpha} 
\bibliography{refs} 

\end{document}